\pdfoutput=1 
\documentclass[11pt,a4paper]{article}

\usepackage[left=2.5cm,right=2.5cm,top=2.5cm,bottom=2.5cm,includefoot]{geometry}
\usepackage{xcolor}
\usepackage[small]{titlesec}
\titlelabel{\thetitle.\quad}
\usepackage{setspace}
\usepackage{grffile} 
\usepackage{footmisc}
\DefineFNsymbols{mySymbols}{{\ensuremath\dagger}{\ensuremath\ddagger}\S\P
   *{**}{\ensuremath{\dagger\dagger}}{\ensuremath{\ddagger\ddagger}}}
\setfnsymbol{mySymbols}
\usepackage[title]{appendix}

\usepackage[english]{babel}
\usepackage[utf8]{inputenc}
\usepackage[T1]{fontenc}
\usepackage[protrusion=true,expansion=true]{microtype}
\usepackage{newtxtext}
\usepackage{amsthm}
\usepackage[smallerops]{newtxmath}
\usepackage{dsfont}
\numberwithin{equation}{section}
\DeclareMathAlphabet{\mathcal}{OMS}{cmsy}{m}{n}

\usepackage{booktabs}
\usepackage{enumitem}
\setlist[enumerate]{align=left}
\usepackage{float}
\usepackage{listings}
\usepackage{multirow,multicol}
\usepackage[font={footnotesize}]{caption}
\usepackage{subcaption}
\usepackage{tabularx}
\usepackage{rotating}

\newtheoremstyle{boldtitle}
{}
{}
{}
{}
{\bfseries}
{.}
{.5em}
{{\thmname{#1 }}{\thmnumber{#2}}{\thmnote{ (#3)}}}

\theoremstyle{boldtitle}

\newtheorem{theorem}{Theorem}[section]
\newtheorem{lemma}[theorem]{Lemma}
\newtheorem{definition}[theorem]{Definition}
\newtheorem{proposition}[theorem]{Proposition}
\newtheorem{corollary}[theorem]{Corollary}

\newtheorem{remark}[theorem]{Remark}

\newtheorem{assumption}[theorem]{Assumption}

\usepackage{xr-hyper} 
\usepackage[
    breaklinks,
    colorlinks=true,
    pagebackref=false,
    pdfauthor={Timo Dimitriadis and Sebastian Bayer},
    pdftitle={A Joint Quantile and Expected Shortfall Regression Framework},
    citecolor=blue,
    linkcolor=blue,
    filecolor=red,
    urlcolor=blue
]{hyperref}      
\usepackage{cleveref}
\usepackage[authoryear,round]{natbib} 
\DeclareMathOperator{\ES}{ES}

\usepackage{todonotes}

\begin{document}
\setlength{\abovedisplayskip}{8pt}
\setlength{\belowdisplayskip}{8pt}

\begin{center}
\vspace*{2em}

{\LARGE A Joint Quantile and Expected Shortfall Regression Framework}

\vspace*{1.5em}
\renewcommand*{\thefootnote}{\fnsymbol{footnote}}
{\large Timo Dimitriadis%
   	\footnote{Corresponding author\\\normalfont Email addresses: \href{mailto:timo.dimitriadis@uni-konstanz.de}{timo.dimitriadis@uni-konstanz.de}, \href{mailto:sebastian.bayer@uni-konstanz.de}{sebastian.bayer@uni-konstanz.de}} %
   	and Sebastian Bayer}
\renewcommand*{\thefootnote}{\arabic{footnote}}
\setcounter{footnote}{0}

\vspace*{1em}

\textit{University of Konstanz, Department of Economics, 78457 Konstanz, Germany}

\vspace*{1em}
\textit{This Version: 07.08.2017}
\end{center}

\noindent\rule{\linewidth}{.4pt}
\textbf{Abstract}
\\[.5\baselineskip]
\noindent
We introduce a novel regression framework which simultaneously models the quantile and the Expected Shortfall (ES) of a response variable given a set of covariates.
This regression is based on a strictly consistent loss function for the pair quantile and ES, which allows for M- and Z-estimation of the joint regression parameters.
We show consistency and asymptotic normality for both estimators under weak regularity conditions.
The underlying loss function depends on two specification functions, whose choice affects the properties of the resulting estimators.
We find that the Z-estimator is numerically unstable and thus, we rely on M-estimation of the model parameters.
Extensive simulations verify the asymptotic properties and analyze the small sample behavior of the M-estimator for different specification functions.
This joint regression framework allows for various applications including estimating, forecasting, and backtesting ES, which is particularly relevant in light of the recent introduction of ES into the Basel Accords.
\\[.5\baselineskip]
\textit{Keywords:} Expected Shortfall, Joint Elicitability, Joint Regression, M-estimation, Quantile Regression
\\[-.5\baselineskip]
\noindent\rule{\linewidth}{.4pt}

\section{Introduction}

Measuring and forecasting risks is essential for a variety of academic disciplines.
For this purpose, risk measures which are formally defined as a map (with certain properties) from a space of random variables to a real number, are applied to condense the complex nature of the involved risks to a single number \citep{Artzner1999}.
In the context of financial risk measurement, to date the most commonly used risk measure is the Value-at-Risk (VaR), which is the $\alpha$-quantile of the return distribution. 
Its popularity is mainly due to its simple nature and the fact that up to now, the Basel Accords stipulate its use for the calculation of capital requirements for banks.
Besides being not coherent \citep{Artzner1999}, the main drawback of the VaR is its inability to capture tail risks beyond itself.
This deficiency is overcome by the risk measure Expected Shortfall (ES) at level $\alpha$, which is defined as the mean of the returns which are smaller than the $\alpha$-quantile of the return distribution.
The ES has the desired ability to capture information from the whole left tail of the return distribution, which is particularly important for measuring extreme financial risks.
Over the past few years, ES has increasingly become the object of interest for practitioners, academics, and regulators, especially since its recent introduction into the Basel Accords \citep{Basel2016}.

A major drawback of the ES (regarded as a statistical functional) is that it is not elicitable, which means that there exists no loss function (scoring function, scoring rule) which the ES uniquely minimizes in expectation \citep{Gneiting2011b, Weber2006}.
This result has two main consequences. 
First, consistent ranking of competing forecasts for the ES based on such a loss function is infeasible.
Second, and more substantial for this paper, modeling the conditional ES given a set of covariates through a regression model without specifying the full conditional distribution is infeasible since estimation of the regression parameters through M-estimation requires such a loss function.
Consequently, and in contrast to quantile regression (which can be used to model the VaR), to date, there exists no such regression framework which models the ES based on a set of covariates.

\cite{Nadarajah2014} provide an overview of estimation methods for the ES.
However, the reviewed approaches are only applicable for univariate data and not suitable for estimating the conditional ES based on covariates such as in mean and quantile regression.
Nevertheless, there are some approaches for the ES which incorporate explanatory variables through indirect estimation procedures.
\citet{Taylor2008b} proposes an implicit approach for forecasting ES using exponentially weighted quantile regression and \citet{Taylor2008} introduces a procedure based on expectile regression and a relationship between the ES and expectiles.
\citet{Taylor2017} suggests a joint modeling technique for the quantile and the ES based on maximum likelihood estimation of the asymmetric Laplace distribution.
\cite{Barendse2017} proposes generalized method of moments (GMM) estimation for a regression framework for the interquantile expectation.

Even though the ES is not elicitable stand-alone, \citet{Fissler2016} show in their seminal paper that the quantile (the VaR) and the ES are jointly elicitable by introducing a class of joint loss functions, whose expectation is minimized by these two functionals.
This joint elicitability result and the associated class of loss functions gives rise to a growing literature in both, joint estimation \citep{Zwingmann2016} and in joint forecast evaluation \citep{Acerbi2014,Fissler2016b,Nolde2017,Ziegel2017} for the risk measures VaR and ES.

In this paper, we utilize the class of loss functions of \citet{Fissler2016} for the introduction of a novel simultaneous regression framework for the quantile and the ES and propose both, an M- and a Z-estimator for the joint regression parameters.
These strictly consistent loss functions facilitate the opportunity to introduce M- and Z-estimation of the regression parameters without specifying the full conditional distribution of the model, as opposed to maximum likelihood estimation.
We show consistency and asymptotic normality for both estimators under weak regularity conditions which are typical for such a regression framework.
To the best of our knowledge, we are the first to propose such a joint regression framework for the quantile and the ES together with the joint M- and Z-estimation and the associated results of consistency and asymptotic normality.
Furthermore, we are the first to propose a joint semiparametric regression framework for two different functionals based on joint M-estimation without specifying the full conditional distribution.

The employed joint loss function, the estimating equations (for the Z-estimator) and the resulting parameter estimates depend on two specification functions, which can be chosen from some class of functions.
Even though consistency and asymptotic normality hold for all applicable choices of these specification functions, they affect the necessary moment conditions, the resulting asymptotic covariance matrices of the estimators, the numerical stability of the optimization algorithm, and the computation times.
We discuss the choice of these functions in a theoretical context with respect to asymptotic efficiency and necessary regularity conditions, and with respect to the numerical properties of the optimization algorithm. 

The estimation of the asymptotic covariance matrix imposes some difficulties.
The first occurs in the estimation of the density quantile function, analogous to quantile regression \citep[cf.][]{Koenker2005book} and thus, we utilize estimation procedures stemming from this literature.
The second issue is the estimation of the variance of the negative quantile residuals conditional on the covariates, a nuisance quantity which is new to the literature.
We introduce several estimators for this quantity which are able to cope with limited sample sizes and which can model the dependency of the negative quantile residuals on the covariates.
Furthermore, we estimate the covariance matrix using the bootstrap.
For ease of application, we provide an R package \citep{ESRegPackge} which contains the implementation of the M- and Z-estimator. 
The user can choose the specification functions, the numerical optimization procedure and the estimation method for the covariance matrix of the parameter estimates.

We conduct a Monte-Carlo simulation study where we consider three data generating processes with different properties.
We numerically verify consistency and asymptotic normality of the M-estimator  for a range of different choices of the specification functions.
Furthermore, we find that the Z-estimator is numerically unstable due to the redescending nature of the utilized estimating equations and consequently, we rely on M-estimation of the regression parameters. 
Moreover, we find that the performance of the M-estimator strongly depends on the specification functions, where choices resulting in positively homogeneous loss functions \citep{Nolde2017, Efron1991} lead to a superior performance in terms of asymptotic efficiency, computation times, and mean squared error of the estimator.

This joint regression technique for the quantile and ES has a wide range of potential applications as it generalizes quantile regression to the pair consisting of the quantile and the ES.
In the context of financial risk management, it opens up the possibility to extend the existing applications of quantile regression on VaR in the financial literature to ES, such as e.g. in \cite{Chernozhukov2001}, \cite{Engle2004}, \cite{KoenkerXiao2006}, \cite{Gaglianone2011}, \cite{Halbleib2012}, \cite{Komunjer2013}, \cite{Xiao2015} and \cite{Zikes2016}.
Such estimation, forecasting, and backtesting methods for the ES are particularly sought-after in light of the recent shift from VaR to ES in the Basel Accords.
As an illustration, we present an empirical application where we use our regression framework to jointly forecast VaR and ES based on the realized volatility.

The rest of the paper is organized as follows. 
In \Cref{sec::Methodology}, we introduce the joint regression framework, the underlying regularity conditions together with the asymptotic properties of our estimators and discuss the choice of the specification functions.
\Cref{sec::ModelEstimation} provides details on the numerical implementation of the estimators and on the estimation of the asymptotic covariance matrix.
\Cref{sec::Simulations} presents an extensive simulation study and Section \ref{sec::EmpiricalApplication} contains an empirical application.
\Cref{sec::Conclusion} provides concluding remarks.
The proofs are deferred to Appendices \ref{sec::Proofs} and \ref{sec::TechnicalResults}.

\section{Methodology}
\label{sec::Methodology}

\subsection{The Joint Regression Framework}
\label{sec::JointRegressionFramework}

Following \citet{Lambert2008}, \cite{Gneiting2011b} and \citet{Fissler2016}, we introduce the concept of (multivariate) $p$-elicitability.
We consider a random variable $Z: \Omega \to \mathbb{R}^d$, defined on some complete probability space $\big( \Omega, \mathcal{F}, P \big)$, a class of distributions $\mathcal{P}$ on $\mathbb{R}^d$, equipped with the Borel $\sigma$-field and a functional $T: \mathcal{P} \to D$ with its domain of action $D \subseteq \mathbb{R}^p, p \in \mathbb{N}$.
We call an integrable loss function $\rho: \mathbb{R}^d \times D \to \mathbb{R}$ \textit{strictly consistent} for the functional $T$ relative to the class of distributions $\mathcal{P}$, if $T$ is the unique minimizer of $\mathbb{E} \big[ \rho(Z,\cdot) \big]$ for all distributions $F \in \mathcal{P}$, where $F$ is the distribution of $Z$.
Furthermore, we call a $p$-dimensional functional $T$ $p$-\textit{elicitable} relative to the class $\mathcal{P}$, if there exists a loss function $\rho$ which is strictly consistent for $T$ relative to $\mathcal{P}$.
If the dimension $p$ is clear from the context, we simply call the functional elicitable instead of $p$-elicitable.

Given the generalized $\alpha$-quantile $Q_\alpha(Z) = F^{-1}(\alpha) = \operatorname{inf} \big\{ z \in \mathbb{R}: F(z) \ge \alpha \big\}$ for some $\alpha \in (0,1)$, the ES of the random variable $Z$ at level $\alpha$ is defined as $\ES_\alpha (Z) = \frac{1}{\alpha} \int_0^\alpha Q_u(Z) \, \mathrm{d}u$.
If the distribution function of $Z$ is continuous at its $\alpha$-quantile, this definition can be simplified to the conditional tail expectation $\ES_\alpha (Z) = \mathbb{E} \big[ Z \, \big| \, Z \le Q_\alpha(Z) \big]$.
\citet{Gneiting2011b} shows that the ES is not  $1$-elicitable with respect to any class $\mathcal{P}$ of probability distributions on intervals $I \subseteq \mathbb{R}$, which contain measures with finite support or finite mixtures of absolutely continuous distributions with compact support (see also \citealp{Weber2006}).
This result has several consequences for the risk measure ES. First, consistent and meaningful ranking of competing forecasts for the functional ES is infeasible.
Second, and more consequential for this work, estimating the parameters of a stand-alone regression model for the functional ES in the sense that $\ES_\alpha(Y|X) = X' \theta^e$ by means of M-estimation, i.e. by minimizing some strictly consistent loss function, is infeasible.
Even though the ES is not $1$-elicitable, \citet{Fissler2016} show that the pair consisting of the ES and the quantile at common probability level $\alpha$ is $2$-elicitable relative to the class of distributions with finite first moments and unique $\alpha$-quantiles and they characterize the full class of strictly consistent loss functions for this pair subject to some regularity conditions.
Since the definition of the ES already depends on the respective quantile, the fact that the ES is only elicitable jointly with the quantile is not surprising.

We utilize this joint elicitability result for the introduction of a new joint regression framework for the quantile and the ES where the aforementioned class of strictly consistent loss functions serves as the basis for the M-estimation of the joint regression parameters.
For this, let $Y: \Omega \to \mathbb{R}$ and $X: \Omega \to \mathbb{R}^k$ be random variables defined on the same probability space $\big( \Omega, \mathcal{F}, P \big)$ as above. Henceforth, the transpose of $X$ will be denoted by $X'$, the cumulative distribution function of $Y$ given $X$ by $F_{Y|X}$ and the conditional density function by $f_{Y|X}$.
For a $k$-times differentiable real-valued function $G: \mathbb{R} \to \mathbb{R}$, we denote the $k$-th derivative by $G^{(k)}(\cdot)$.

\begin{assumption}[The joint regression model] 
	\label{ass::Model}
	The regression framework which jointly models the conditional quantile and ES of $Y$ given $X$ for some fixed level $\alpha \in (0,1)$ is given by
	\begin{align}\label{eqn::JointRegressionEquation}
		Y = X' \theta^q_0 + u^q \qquad \text{and} \qquad Y = X' \theta^e_0 + u^e,
	\end{align}
	where $Q_{\alpha}(u^q|X)  = 0$ and $\ES_{\alpha}(u^e|X) = 0$. The model is parametrized by $\theta_0 = (\theta_0^{q\prime},\theta_0^{e\prime})' \in \Theta \subset \mathbb{R}^{2k}$, where the parameter space $\Theta$ is compact with nonempty interior, $\operatorname{int}(\Theta) \not= \emptyset$.
\end{assumption}

We propose both, an M-estimation and a Z-estimation procedure for the compound regression parameter vector $\theta_0$.
For the M-estimation, we adapt the class of strictly consistent joint loss functions\footnote{One can interpret the structure of this loss function as follows \citep{Fissler2016b}: The first summand in (\ref{eqn::regressionrhofunction}) is a strictly consistent loss function for the quantile \citep{Gneiting2011b} and hence only depends on the quantile, whereas the second summand cannot be split into a part depending only on the quantile  and one depending only on the ES. This illustrates the fact that the ES itself is not 1-elicitable, but 2-elicitable together with the respective quantile.} for the quantile and ES as given in \cite{Fissler2016} such that it can be used in a regression framework,	
\begin{align} 
	\begin{aligned}
	\label{eqn::regressionrhofunction}          
	\rho(Y,X,\theta) = \;   &\big(\mathds{1}_{\{Y \le X' \theta^q\}} - \alpha \big) G_1(X' \theta^q) - \mathds{1}_{\{Y \le X' \theta^q\}} G_1(Y) \\
	& \; + G_2(X' \theta^e) \left( X' \theta^e - X' \theta^q + \frac{(X' \theta^q - Y) \mathds{1}_{\{Y \le X' \theta^q\}}}{\alpha}  \right) - \mathcal{G}_2(X' \theta^e) + a(Y),
	\end{aligned}	
\end{align}	
where the function $G_1$ is twice continuously differentiable, $\mathcal{G}_2$ is three times continuously differentiable, $\mathcal{G}_2^{(1)} = G_2$, $G_2$ and $G_2^{(1)}$ are strictly positive, $G_1$ is increasing and $a$ and $G_1$ are integrable.
We discuss the choice of the \textit{specification functions} $G_1$ and $\mathcal{G}_2$ in a theoretical context in Section \ref{sec::ChoiceOfTheGfunctions} and by their numerical performance in Section \ref{sec::MCStudyChoiceOfTheGfunctions}.
The corresponding ($\rho$-type) M-estimator is defined by a sequence $\hat \theta_{\rho,n}$, such that $\hat \theta_{\rho,n} = \operatorname{argmin}_{\theta \in \Theta} \frac{1}{n} \sum_{i=1}^{n} \rho(Y_i,X_i,\theta)$.

Instead of minimizing some objective function $\rho(Y,X,\theta)$ such as in (\ref{eqn::regressionrhofunction}), we can also define the corresponding Z-estimator (or $\psi$-type M-estimator), which sets a vector of estimating equations (moment conditions), denoted by $\psi(Y,X,\theta)$, to zero.
More generally, it suffices that these estimating equations converge to zero almost surely.
Formally, the Z-estimator is a sequence $\hat \theta_{\psi,n}$, such that $\frac{1}{n} \sum_{i=1}^{n} \psi(Y_i,X_i,\hat \theta_{\psi,n}) \to 0$ almost surely, where
\begin{align} 
	\label{eqn::regressionpsifunction}
	\psi(Y,X,\theta)
	= \begin{pmatrix}
	\psi_1(Y,X,\theta) \\
	\psi_2(Y,X,\theta) 
	\end{pmatrix}
	= \begin{pmatrix}
	\frac{1}{\alpha}(\mathds{1}_{\{Y \le X' \theta^q\}} -\alpha) \big( \alpha X G_1^{(1)}(X'\theta^q) + X G_2(X' \theta^e) \big) \vspace{0.1cm} \\
	X G_2^{(1)} ( X'\theta^e) \left( X' \theta^e - X' \theta^q + \frac{1}{\alpha} (X' \theta^q - Y) \mathds{1}_{\{Y \le X' \theta^q\}}  \right)
	\end{pmatrix},
\end{align}	
which is obtained by differentiating\footnote{
	Note that the function $\rho(Y,X,\theta)$, given in (\ref{eqn::regressionrhofunction}) is only differentiable for $Y \not= X'\theta^q$. However, the points of non-differentiability, $Y = X'\theta^q$ form a nullset with respect to the absolutely continuous distribution of $Y$ given $X$.
}
(\ref{eqn::regressionrhofunction}) and where the functions $G_1$ and $G_2$ are given as above.
When the loss function $\rho(Y,X,\theta)$ is continuously differentiable in $\theta$, it is obvious that the M- and Z-estimation approaches are equivalent.
However, in this case the loss function $\rho(Y,X,\theta)$ is not differentiable and $\psi(Y,X,\theta)$ is discontinuous at the points where $Y = X' \theta^q$.
Thus, we treat these two estimation approaches as different estimators and show their asymptotic behavior separately.

\subsection{Asymptotic Properties}
\label{sec::AsymptoticProperties}
In this section, we present the asymptotic properties of the M- and Z-estimator of the regression parameters. Consistency and asymptotic normality hold under the following set of weak regularity conditions, which are natural for this regression framework.

\setlist[enumerate]{
	labelsep=8pt,
	labelindent=0.5\parindent,
	itemindent=0pt,
	leftmargin=50pt,
	before=\setlength{\listparindent}{-\leftmargin},
}

\begin{assumption}[Regularity Conditions] \label{ass::GeneralAssumptionsConsistency}
	$ $
	\begin{enumerate}[label= ($\mathcal{A}$-\arabic*)]
		\item 
		\label{RegCond::ConsistencyConditionalDistribution}
		The data $(Y_i,X_i)$ for $i = 1, \dots ,n$ is an iid series of random variables, distributed such as $(Y,X)$ given above.
		Furthermore, the conditional distribution $F_{Y|X}$ has finite second moments and is absolutely continuous with probability density function $f_{Y|X}$, which is strictly positive, continuous and bounded in a neighbourhood of the true conditional quantile, $X'\theta^q_0$.

		\item 
		\label{RegCond::ConsistencyFullRankCondition}
		The matrix $\mathbb{E} \big[ X X' \big]$ is positive definite.

		\item 
		\label{RegCond::GFunctions}
		The functions $\rho(Y,X,\theta)$ and $\psi(Y,X,\theta)$ are given as in (\ref{eqn::regressionrhofunction}) and (\ref{eqn::regressionpsifunction}), where the function $G_1$ is twice continuously differentiable, $\mathcal{G}_2$ is three times continuously differentiable, $\mathcal{G}_2^{(1)} = G_2$, $G_2$ and $G_2^{(1)}$ are strictly positive, $G_1$ is increasing and $a$ and $G_1$ are integrable.

	\end{enumerate}
\end{assumption}

\begin{remark}[Finite Moment Conditions]
	We further have to assume that certain moments of $X$ are finite. For the sake of space, we specify the Finite Moment Conditions \ref{MomCond::ConsistencyPsi} - \ref{MomCond::AsymptoticNormalityRho} in  Appendix \ref{sec::GeneralMomentConditions}.
	Note that these general moment conditions simplify substantially for sensible choices of the specification functions $G_1$ and $\mathcal{G}_2$ as further outlined in Section \ref{sec::ChoiceOfTheGfunctions}. 
\end{remark}

Assumption \ref{RegCond::ConsistencyConditionalDistribution} is a combination of typical regularity conditions of mean and quantile regression.
Absolute continuity of $F_{Y|X}$ with a strictly positive, bounded and continuous density function in a neighborhood of the true conditional quantile is also imposed for the asymptotic theory of quantile regression.
Existence of the conditional moments of $Y$ given $X$ is subject to the conditions of mean regression and is included in our regularity conditions since ES is a truncated mean.
The positive definiteness (full rank condition) in \ref{RegCond::ConsistencyFullRankCondition} is common for any regression design with stochastic regressors in order to exclude perfect multicollinearity of the regressors.
The conditions for the specification functions $G_1$ and $\mathcal{G}_2$ in \ref{RegCond::GFunctions} mainly originate from the conditions for the joint elicitability of the quantile and ES in \citet{Fissler2016}. 
Differentiability of these functions is required in this setup for obtaining the estimating equations and for the differentiations in the computation of the asymptotic covariance in Theorem \ref{thm::AsymptoticNormalityPsi} and Theorem \ref{thm::AsymptoticNormalityRho}.
The existence of certain moments of the explanatory variables as in conditions \ref{MomCond::ConsistencyPsi} - \ref{MomCond::AsymptoticNormalityRho} in Appendix \ref{sec::GeneralMomentConditions} is also standard in any regression design relying on stochastic regressors.
Even though compactness of the parameter space $\Theta$ in Assumption \ref{ass::Model} generally simplifies the proofs, in this setup it is crucial for consistency of the Z-estimator as the estimating equations $\psi_2$ are redescending to zero for many reasonable choices of the $G_2$ function such as e.g. the choices resulting in positively homogeneous loss functions. For details on this, we refer to Section \ref{sec::Optimization}.

\begin{theorem} \label{thm::ConsistencyPsi}
	Assume that Assumption \ref{ass::Model}, Assumption \ref{ass::GeneralAssumptionsConsistency} and the Moment Conditions \ref{MomCond::ConsistencyPsi} in Appendix \ref{sec::GeneralMomentConditions} hold true. Then, for every sequence $\hat \theta_{\psi,n} \in \Theta$ satisfying $\frac{1}{n} \sum_{i=1}^{n} \psi(Y_i,X_i,\hat \theta_{\psi,n}) \stackrel{a.s.}{\longrightarrow} 0$, it holds that $\hat \theta_{\psi,n} \stackrel{a.s.}{\longrightarrow} \theta_0$.
\end{theorem}

\begin{theorem}
	\label{thm::ConsistencyRho}
	Assume that Assumption \ref{ass::Model}, Assumption \ref{ass::GeneralAssumptionsConsistency} and the Moment Conditions \ref{MomCond::ConsistencyRho} in Appendix \ref{sec::GeneralMomentConditions} hold true. Then, for every sequence $\hat \theta_{\rho,n} \in \Theta$ such that
	$\frac{1}{n} \sum_{i=1}^{n} \rho(Y_i,X_i,\hat \theta_{\rho,n}) \le \frac{1}{n} \sum_{i=1}^{n} \rho(Y_i,X_i, \theta_0) + o_{P}(1)$,
	it holds that $\hat \theta_{\rho,n} \stackrel{\mathbb{P}}{\longrightarrow} \theta_0$.
\end{theorem}

\begin{theorem} \label{thm::AsymptoticNormalityPsi}
	Assume that Assumption \ref{ass::Model}, Assumption \ref{ass::GeneralAssumptionsConsistency} and the Moment Conditions \ref{MomCond::AsymptoticNormalityPsi} in Appendix \ref{sec::GeneralMomentConditions} hold true. Then, for every sequence $\hat \theta_{\psi,n} \in \Theta$ satisfying $\frac{1}{\sqrt{n}} \sum_{i=1}^{n} \psi(Y_i,X_i,\hat \theta_{\psi,n}) \stackrel{\mathbb{P}}{\longrightarrow} 0$, it holds that
	\begin{align}
	\sqrt{n} \big( \hat \theta_{\psi,n} - \theta_0 \big) \stackrel{d}{\longrightarrow} \mathcal{N} \left(0,\Lambda^{-1} C \Lambda^{-1} \right),
	\end{align}
	with
	\begin{align}
		\Lambda = \begin{pmatrix} \Lambda_{11} & 0 \\ 0 & \Lambda_{22}  \end{pmatrix} \qquad \text{and} \qquad	C = \begin{pmatrix} C_{11} & C_{12} \\ C_{21} & C_{22} \end{pmatrix},  
	\end{align}
	where
	\begin{align}
		\Lambda_{11} &= \frac{1}{\alpha} \mathbb{E} \left[ (XX')  f_{Y|X}(X' \theta^q_0) \bigl( \alpha G_1^{(1)}(X' \theta^q_0) + G_2(X' \theta^e_0) \bigr) \right], \label{eqn::Lambda11} \\
		\Lambda_{22} &= \mathbb{E} \big[ (X X') G_2^{(1)} (X' \theta^e_0) \big], \label{eqn::Lambda22} \\
		C_{11}	&=  \frac{1-\alpha}{\alpha} \mathbb{E} \left[ (X X') \big( \alpha G_1^{(1)}(X' \theta^q_0) + G_2(X' \theta^e_0) \big)^2 \right], \label{eqn::AsyCovMatrixC11} \\
		C_{12} &= C_{21} = \frac{1-\alpha}{\alpha}  \mathbb{E} \left[(XX') \big( X' \theta^q_0 - X' \theta^e_0 \big) \big(\alpha G_1^{(1)}(X' \theta^q_0) + G_2(X' \theta^e_0) \big)  G_2^{(1)}(X' \theta^e_0)	 \right], \label{eqn::AsyCovMatrixC12} \\
		C_{22} &= \mathbb{E} \left[ (X X') \big({G_2^{(1)}}(X' \theta^e_0) \big)^2 \left(\frac{1}{\alpha} \operatorname{Var} \big(Y - X' \theta^q_0 \big| Y \le X' \theta^q_0, X \big) + \frac{1-\alpha}{\alpha} \big(X' \theta^q_0 - X' \theta^e_0 \big)^2 \right)  \right]. \label{eqn::AsyCovMatrixC22}
	\end{align}
\end{theorem}

\begin{theorem}
	\label{thm::AsymptoticNormalityRho}
	Assume that Assumption \ref{ass::Model}, Assumption \ref{ass::GeneralAssumptionsConsistency} and the Moment Conditions \ref{MomCond::AsymptoticNormalityRho} in Appendix \ref{sec::GeneralMomentConditions} hold true. Then, for every sequence $\hat \theta_{\rho,n} \in \Theta$ such that
	$\frac{1}{n} \sum_{i=1}^{n} \rho(Y_i,X_i,\hat \theta_{\rho,n}) \le \inf_{\theta \in \Theta} \frac{1}{n} \sum_{i=1}^{n} \rho(Y_i,X_i, \theta) + o_{P}(n^{-1})$,
	it holds that
	\begin{align}
	\sqrt{n} \big( \hat \theta_{\rho,n} - \theta_0 \big) \stackrel{d}{\longrightarrow} \mathcal{N}\big(0, \Lambda^{-1} C \Lambda^{-1} \big),
	\end{align}
	where the matrices $\Lambda$ and $C$ are given as in Theorem \ref{thm::AsymptoticNormalityPsi}.
\end{theorem}

\begin{remark}[Quantile Regression]
	Notice that the asymptotic covariance matrix of the quantile-specific parameter estimates $\hat \theta^q$ is given by $\alpha (1-\alpha) D_1^{-1} D_0 D_1^{-1}$, where
	\begin{align}
	D_1 &= \mathbb{E} \left[ (XX')  f_{Y|X}(X' \theta^q_0)  \bigl( \alpha G_1^{(1)}(X' \theta^q_0) + G_2(X' \theta^e_0) \bigr) \right] \qquad \text{and} \\
	D_0 &= \mathbb{E} \left[ (XX') \bigl( \alpha G_1^{(1)}(X' \theta^q_0) + G_2(X' \theta^e_0) \bigr)^2 \right].
	\end{align}
	This simplifies to the covariance matrix of quantile regression parameter estimates by setting
	$G_1(z) = z$ and $G_2(z) = 0$, which means ignoring the ES-specific part of our loss function and estimating equations. This demonstrates that the quantile regression method is nested in our regression procedure, also in terms of its asymptotic distribution.
\end{remark}

\begin{remark}[Asymptotic Covariance of the ES and the Oracle Estimator]
	The ES-specific part of the asymptotic covariance is mainly governed by the term $C_{22}$, which depends on the quantity
	\begin{align}
	\label{eqn::InterpretationESRegCovariance}
	\frac{1}{\alpha} \operatorname{Var} \big(Y - X' \theta^q_0 \big| Y \le X' \theta^q_0, X \big) + \frac{1-\alpha}{\alpha} \big(X' \theta^q_0 - X' \theta^e_0 \big)^2
	= \frac{1}{\alpha^2} \operatorname{Var} \left( \left. (Y - X'\theta^q_0) \mathds{1}_{\{ Y \le X' \theta^q_0\}} \right| X \right).
	\end{align} 
	It is reasonable that the asymptotic covariance of ES regression parameters depends on the truncated variance of $Y$ given $X$ as the asymptomatic covariance of mean regression parameters is driven by the conditional (non-truncated) variance of $Y$ given $X$.
	The second term $\big(X' \theta^q_0 - X' \theta^e_0 \big)^2$ in (\ref{eqn::InterpretationESRegCovariance}) is included since the ES represents a truncated mean where the truncation point itself is a statistical functional (the quantile).
	In comparison, we consider an oracle M-estimator for the ES-specific regression parameters $\theta^e$, given by the loss function
	\begin{align}
		\label{eqn::ESRegressionOracleLoss}
		\rho_{\text{Oracle}}(Y,X,\theta^e) &= (Y - X'\theta^e)^2 \mathds{1}_{\{ Y \le X' \theta^q_0\}},
	\end{align}
	where we assume that the true quantile regression parameters $\theta^q_0$ are known.
	The resulting asymptotic covariance is given by
	\begin{align}
		\operatorname{AVar} \left( \widehat \theta^e_{\text{Oracle}} \right)
		= \frac{1}{\alpha} \mathbb{E} \big[XX'\big]^{-1} \cdot \mathbb{E} \left[ (XX') \operatorname{Var} \left( \left. Y - X'\theta^e_0 \right| Y \le X'\theta^q_0, X \right) \right] \cdot \mathbb{E} \big[XX'\big]^{-1},
	\end{align}	
	which shows that the additional term $\big(X' \theta^q_0 - X' \theta^e_0 \big)^2$ is not included for this estimator with fixed truncation point $X'\theta^q_0$.
\end{remark}

\begin{remark}[Joint Estimation of the Sample Quantile and ES]
	We can use this regression framework to jointly estimate the quantile and ES of an identically distributed sample $Y_1 , \dots , Y_n$ by regressing on a constant only.
	The asymptotic covariance matrix given in Theorem \ref{thm::AsymptoticNormalityPsi} and Theorem \ref{thm::AsymptoticNormalityRho} then simplifies to $\Sigma$ with components
	\begin{align}
	\Sigma_{11} &= \frac{\alpha (1-\alpha)}{f_{Y}^2(\theta^q_0)}, \label{eqn::AsymptoticCovQuantileEstimates} \\
	\Sigma_{12} &= \Sigma_{21} = (1-\alpha) \frac{\theta^q_0 - \theta^e_0}{f_Y(\theta^q_0)},  \\
	\Sigma_{22} &= \frac{1}{\alpha} \operatorname{Var}(Y-\theta^q_0 | Y \le \theta^q_0) + \frac{1-\alpha}{\alpha} (\theta^q_0 - \theta^e_0)^2, \label{eqn::CovarianceIidESMinEstimator}
	\end{align}
	where $\theta^q_0$ and $\theta^e_0$ are the true quantile and ES of $Y$.
	The same result is obtained by \cite{Zwingmann2016}, who further allow for a distribution function for $Y$ which is not differentiable at the quantile with strictly positive derivative.
	Notice that in this simplified case without covariates, the asymptotic covariance matrix is independent of the specification functions $G_1$ and $\mathcal{G}_2$ used in the loss function and in the estimating equations.
	Furthermore, (\ref{eqn::AsymptoticCovQuantileEstimates}) implies that quantile estimates stemming from our joint estimation procedure have the same asymptotic efficiency as quantile estimates stemming from minimizing the generalized piecewise linear loss \citep{Gneiting2011b} and as sample quantiles (cf. \citealp{Koenker2005book}).
	The same holds true for the efficiency of the sample ES estimators (based on the sample quantile) of \citet{Brazauskas2008} and \cite{Chen2008}.
\end{remark}

\begin{remark}[Pseudo-$\boldsymbol{R^2}$ and the choice of $\boldsymbol{a(Y)}$]
	By choosing $a(Y) = \alpha G_1(Y) + \mathcal{G}_2(Y)$

	in (\ref{eqn::regressionrhofunction}), we can guarantee non-negative losses $\rho(Y,X,\theta) \ge 0$.
	This choice enables us to define a pseudo-$R^2$ for our joint regression framework in the sense of \cite{Koenker1999},
	\begin{align}
	R^{QE} = 1- \frac{\rho(Y,X,\hat \theta)}{\rho(Y,X,\tilde \theta)},
	\end{align}
	where $\hat \theta$ denotes the parameter estimates of the full regression model and $\tilde \theta$ denotes the parameter estimates of a regression model restricted to an intercept term only.
	However, this choice of $a(Y)$ comes at the cost of more restrictive moment conditions, since we need to impose that $\mathbb{E} \big[ G_1(Y) + \mathcal{G}_2(Y) \big] < \infty$.	
\end{remark}

\subsection{Choice of the Specification Functions}
\label{sec::ChoiceOfTheGfunctions}

The loss functions and the estimating equations given in (\ref{eqn::regressionrhofunction}) and (\ref{eqn::regressionpsifunction}) depend on two specification functions, $G_1$ and $\mathcal{G}_2$ (with derivative $G_2$), which have to fulfill the regularity conditions \ref{RegCond::GFunctions} in Assumption \ref{ass::GeneralAssumptionsConsistency}.
\cite{Fissler2016b} already mention the feasible choices $G_1(z) = 0$, $G_1(z) = z$, $G_2(z) = \exp(z)$ and $G_2(z) = \exp(z) / \big( 1+\exp(z) \big)$ in order to show that this class is non-empty.
In contrast to the loss functions of mean, quantile and expectile regression, there is no natural choice for these specification functions for the quantile and ES yet \citep{Nolde2017}.
However, as the choice of these functions strongly influences the performance of our regression procedure in terms of its asymptotic efficiency, the necessary moment conditions of the regressors and the numerical performance of the optimization algorithm, we discuss sensible selection criteria in the following.

\cite{Efron1991} and \cite{Nolde2017} argue that for M-estimation of regression parameters it is crucial that the utilized loss function is positively homogeneous of some order $b \in \mathbb{R}$ in the sense that
\begin{align}
	\rho(cY,X,c\theta) = c^b 	\rho(Y,X,\theta)
\end{align}
for all $c > 0$.
This is an important property for loss functions since the ordering of the losses should be independent of the unit of measurement, e.g. the currency we measure the prices and risk forecasts with. 
Loss functions following this property guarantee that we can change the scaling and still obtain the same optima and consequently the same parameter estimates.
For the pair consisting of the quantile and the ES, \citet{Nolde2017} characterize the full class of positively homogeneous\footnote{For $b=0$, only the loss differences are positively homogeneous. However, the ordering of the losses is still unaffected under this slightly weaker property.} loss functions of order $b$ for the case where we restrict the domain of $\mathcal{G}_2$, i.e. the conditional ES to the negative real line\footnote{Since the conditional ES of financial assets for small probability levels is always negative, this is no critical restriction.
However, for the numerical parameter estimation, we have to restrict the parameter space $\Theta$ such that $X_i' \theta^e < 0$ for all $\theta \in \Theta$ and for all $X_i$ in the underlying sample.
For details on this, we refer to Section \ref{sec::Optimization}.},
\begin{alignat}{3}
	\label{eqn::SpecFuncPosHomo1}
	&b < 0:	\qquad		
	&&G_1(z) = -c_0, \qquad 
	&& \mathcal{G}_2(z) = c_1 (-z)^b + c_0, \\
	\label{eqn::SpecFuncPosHomo2}
	&b = 0: \qquad		
	&&G_1(z) = d_0 \mathds{1}_{\{ z \le 0 \}}  + d_0' \mathds{1}_{\{ z > 0 \}}, \quad  &&\mathcal{G}_2(z) = - c_1 \log(-z) + c_0,  \\
	\label{eqn::SpecFuncPosHomo3}
	&b \in (0,1): \qquad
	&&G_1(z) = \big( d_1 \mathds{1}_{\{ z \le 0 \}}  + d_1' \mathds{1}_{\{ z > 0 \}} \big) |z|^b - c_0, \quad 
	&&\mathcal{G}_2(z) = -c_1 (-z)^b + c_0,
\end{alignat}
for some constants $c_0, d_0, d_0' \in \mathbb{R}$ with $d_0 \le d_0'$, $d_1, d_1' \ge 0$ and $c_1 > 0$. There are no positively homogeneous loss functions for the cases $b \ge 1$.
Our numerical simulations show that there is no gain in efficiency or numerical accuracy by deviating from the choice $G_1(z) = 0$ (see also \citealp{Fissler2016b, Nolde2017, Ziegel2017}), which is also consistent with the homogeneity result. Consequently, we use $G_1(z) = 0$ in the following.

A different natural guiding principle for selecting the specification functions is induced by choosing $\mathcal{G}_2$ (and $G_1$) such that the moment conditions \ref{MomCond::ConsistencyPsi} - \ref{MomCond::AsymptoticNormalityRho} in Appendix \ref{sec::GeneralMomentConditions} are as least restrictive and as parsimonious as possible. 
For instance, choosing $\mathcal{G}_2$ such that $G_2$ and its first and second derivatives are bounded functions (and $G_1(z) = 0$) results in the moment condition $\mathbb{E} \left[ ||X||^5 + ||X||^4 \mathbb{E}\big[ |Y| \big| X \big] + ||X||^3 \mathbb{E}\big[ Y^2 \big| X \big] + | a(Y) | \right] < \infty$.
This motivates the usage of bounded functions\footnote{
	Note that the positively homogeneous loss functions exhibit unbounded  $\mathcal{G}_2$ functions. However, as the function $\mathcal{G}_2(z)$ does not grow faster than linear as $z$ tends to infinity, the resulting finite moment conditions are not too restrictive.
} 
for $G_2$ such as e.g. the second example of \cite{Fissler2016b}, $G_2(z) = \exp(z) / \big( 1+\exp(z) \big)$, which is the distribution function of the standard logistic distribution.
Further examples of bounded $G_2$ functions include the distribution functions of absolutely continuous distributions on the real line.
In the simulation study in Section \ref{sec::MCStudyChoiceOfTheGfunctions}, we compare the performance of different specification functions in terms of mean squared error, asymptotic efficiency of the estimator and computation times.

\section{Numerical Estimation of the Model}
\label{sec::ModelEstimation}

In this section, we discuss the difficulties one encounters and the solutions we propose for estimating the joint regression model.
Section \ref{sec::Optimization} illustrates the numerical optimization procedure we employ for estimating the regression parameters and Section \ref{sec::CovEstimation} discusses different estimation methods for the covariance matrix of the estimator.

\subsection{Optimization}
\label{sec::Optimization}

Theorem \ref{thm::AsymptoticNormalityPsi} and Theorem \ref{thm::AsymptoticNormalityRho} imply that both, M-estimation and Z-estimation of the regression parameters $\theta$ have the same asymptotic efficiency and consequently, we discuss these estimation approaches in terms of their numerical performance in the following.
The numerical implementation of the Z-estimator relies on root-finding of the estimating equations given in (\ref{eqn::regressionpsifunction}), which we implement as in GMM-estimation by minimizing the inner product $\sum_i \psi(Y_i, X_i, \theta)' \cdot \sum_i \psi(Y_i, X_i, \theta)$.
However, the estimating equations are redescending to zero for many attractive choices of $\mathcal{G}_2$ in the sense that $\psi_2(Y,X,\theta) \to 0$ for $X'\theta^e \to -\infty$.
Consequently, for $\theta$ such that $\theta^q = \theta^q_0$ and $X'\theta^e \to - \infty$, we get the same minimal value of the Z-estimation objective function $\sum_i \psi(Y_i, X_i, \theta)' \cdot \sum_i \psi(Y_i, X_i, \theta)$ as for the true regression parameters $\theta_0$.
Thus, the Z-estimator is numerically unstable and diverges in many setups.

Consequently, we rely on M-estimation of the regression parameters in the following.
As the loss functions given in (\ref{eqn::regressionrhofunction}) are not differentiable and non-convex for all applicable choices of the specification functions \citep{FisslerThesis2017}, we apply a derivative-free global optimization technique.
More specifically, we use the Iterated Local Search (ILS) meta-heuristic of \citet{Lourenco2003}, which successively refines the parameter estimates by repeated optimizations with iteratively perturbed starting values.
Our exact implementation consists of the following steps.
First, we obtain starting values for $\theta^q$ and $\theta^e$ from two quantile regressions of $Y$ on $X$ for the probability levels $\alpha$ and $\tilde{\alpha}$, where we choose $\tilde{\alpha}$ such that the  $\tilde{\alpha}$-quantile and the $\alpha$-ES coincide under normality.
Second, using these starting values we minimize the loss function with the derivative-free and robust Nelder-Mead Simplex algorithm \citep{Nelder1965}.
Third, we perturb the resulting parameter estimates by adding normally distributed noise with zero mean and standard deviation equal to the estimated asymptotic standard errors of the initial quantile regression estimates.
Fourth, we re-optimize the model with the perturbed parameter estimates as new starting values.
If the loss is further decreased by this re-optimization, we update the estimates and otherwise, we retain the previous ones.
Fifth, we iterate over the previous two steps until the loss does not decrease in $m=10$ consecutive iterations.
Our numerical experiments indicate that this repeated optimization procedure yields estimates very close to the ones stemming from other global optimization techniques such as e.g. simulated annealing, whereas the major advantage of ILS is the considerably lower computation time.

For the choices of the specification functions which result in positively homogeneous loss functions, we have to restrict the domain of $\mathcal{G}_2$ to the negative real line as already discussed in Section \ref{sec::ChoiceOfTheGfunctions}.
Thus, we have to restrict $\Theta$ such that $X_i'\theta^e < 0$ for all $\theta \in \Theta$ and for all $i = 1, \dots , n$ during the optimization process.
Even though in financial risk management the response variable $Y$ is usually given by financial returns where the true (conditional) ES is strictly negative, there might still be some outliers $X_i$ such that $X_i'\theta^e_0 \geq 0$. 
In such a case, imposing the restriction $X_i'\theta^e < 0$ for all $i = 1, \dots , n$ during the optimization process generates substantially biased estimates for $\theta^e$.
In order to avoid this, we estimate the regression model for the transformed dependent variables $Y - \max(Y)$ for the positively homogeneous loss functions and add $\max(Y)$ to the estimated intercept parameters to undo the transformation\footnote{
	Note that this data transformation changes the average loss function as the applied loss functions are in general not translation invariant. 
	Thus, optimizing the translated loss function can lead to different parameter estimates.
	However, we do not face the risk of obtaining substantially biased estimates in cases where $X_i'\theta_0^e \geq 0$ for some $i \in \{ 1, \dots n \}$.
	Our numerical experiments indicate that the difference between estimating the model for $Y$ and for $Y - \max(Y)$ is small when $X_i'\theta^e_0 < 0$ for all  $i \in \{ 1, \dots n \}$, but can be quite substantial if there is an outlier for $X_i$ such that $X_i'\theta_0^e \geq 0$.
}.

We provide an R package for the estimation of the regression parameters \citep[see][]{ESRegPackge}.
This package contains an implementation of both, the M- and the Z-estimator, where different optimization algorithms can be chosen (ILS, simulated annealing).
The package allows for choosing the specification functions $G_1$ and $\mathcal{G}_2$ and it includes an option to estimate the model either with or without the translation of the dependent variable.
Furthermore, the covariance matrix of the parameter estimates can be estimated either by using the asymptotic theory and the resulting techniques we discuss in the next section, or by using the nonparametric iid bootstrap \citep{Efron1979}.
We recommend applying the M-estimator with the ILS algorithm as this procedure exhibits the best performance in our numerical experiments with respect to accuracy, stability and computation times.

\subsection{Asymptotic Covariance Estimation}
\label{sec::CovEstimation}

While most parts of the asymptotic covariance matrix given in \Cref{thm::AsymptoticNormalityPsi} and \Cref{thm::AsymptoticNormalityRho} are straightforward to estimate, two nuisance quantities impose some difficulties.
The first is the density quantile function $f_{Y|X} (X' \theta^q_0)$,
which is already well investigated in the quantile regression literature.
In particular, we consider the estimators proposed by \citet{Koenker1994}, henceforth denoted by \textit{iid} and by \citet{Hendricks1992}, henceforth denoted by \textit{nid}.
The main difference between these is that the first is based on the assumption that the quantile residuals are independent of the covariates, whereas the second allows for a linear dependence structure.
Both approaches depend on a bandwidth parameter which we choose according to \citet{Hall1988}.

The second nuisance quantity is the variance of the quantile residuals, conditional on the covariates and given that these residuals are negative,
\begin{align}
	\label{eqn::TruncatedVarianceEstimation}
	\operatorname{Var} \big(Y - X' \theta^q_0 \big| Y \le X' \theta^q_0, X \big) = \operatorname{Var} \big(u^q \big|u^q \le 0, X\big).
\end{align} 
Estimation of this quantity is demanding for two reasons. 
First, for very small probability levels which are typical in financial risk management such as e.g. $\alpha = 2.5\%$, the truncation $u^q \le 0$ cuts off all but very few (about $\alpha \cdot n$) observations.
Second, modeling this truncated variance conditional on the covariates $X$ is challenging, especially considering the very small sample sizes. 
Under the assumption of homoscedasticity, i.e. that the distribution of $u^q$ is independent of the covariates $X$, we can simply estimate (\ref{eqn::TruncatedVarianceEstimation}) by the sample variance of the negative quantile residuals and we refer to this estimator as \textit{ind} in the following.

We propose two further estimators which allow for a dependence of the quantile residuals on the covariates.
For this purpose, we assume a location-scale process with linear\footnote{
	This approach can further be generalized by considering more general specifications for the conditional mean and standard deviation. 
	However, our numerical experiments indicate that the estimation accuracy for the asymptotic covariance matrix does not increase by deviating from these linear specifications.
}
specifications of the conditional mean and standard deviation in order to explicitly model the conditional relationship  of $u^q$ on $X$,
\begin{align}
	\label{eqn::residuals_location_scale}
	u^q = X'\zeta + X'\phi \cdot \varepsilon,
\end{align}
for some parameter vectors $\zeta, \phi \in \mathbb{R}^k$ and where $\varepsilon \sim G(0,1)$ follows a zero mean, unit variance distribution, such that $u^q|X \thicksim G \big(X'\zeta, (X'\phi)^2 \big)$ with distribution function $F_G$ and density $f_G$.
As we need to estimate the truncated variance of $u^q$ given $u^q \le 0$, i.e. a truncated variant of $(X'\phi)^2$, one possibility is to estimate (\ref{eqn::residuals_location_scale}) only for those observations where $u^q \leq 0$. 
However, this approach particularly suffers from the very few negative quantile residuals as we need to estimate additional parameters compared to the \textit{ind} approach.

We present a feasible alternative by estimating the parameters $\zeta$ and $\phi$ using all available observations of $u^q$ and $X$ by quasi generalized pseudo maximum likelihood \citep[][Section 8.4.4]{Gourieroux1995} and we obtain the truncated conditional variance by the scaling formula
$\operatorname{Var}\left(u^q | u^q \leq 0, X \right) = \int_{-\infty}^{0} z^2 h(z)\,dz - \left(\int_{-\infty}^{0} z h(z)\,dz\right)^2$, where $h(z) = f_G(z) / F_G(0)$ is the truncated conditional density of $u^q$ given $X$ and $u^q \le 0$.
We propose one parametric estimator, henceforth denoted by \textit{scl-N}, where we assume that the distribution $G$ is the normal distribution and apply a closed-form solution to the scaling formula.
We further propose a semiparametric estimator, henceforth denoted by \textit{scl-sp}, where we estimate the distribution $G$ nonparametrically and then apply the scaling formula for this estimated density by numerical integration.

\section{Simulation Study}	
\label{sec::Simulations}	

In this section, we investigate the finite sample behavior of the M-estimator and verify the asymptotic properties derived in Section \ref{sec::AsymptoticProperties} through simulations.
Furthermore, we compare the performance of different choices for the specification functions and evaluate the precision of the different covariance matrix estimators described in Section \ref{sec::CovEstimation}.

\subsection{Data Generating Process}

In order to assess the numerical properties of estimating the joint regression model, we simulate data from a linear location-scale data generating process (DGP),
\begin{align}\label{eqn::MCStudyDGP}
	Y= X'\gamma + (X'\eta) \cdot v,
\end{align}
where $v \sim F(0,1)$ has zero mean and unit variance, $X = \big(1,X_2,\dots, X_{k} \big)'$ and $\gamma, \eta \in \mathbb{R}^k$.
For this process, the true conditional quantile and ES are linear functions in $X$, given by
\begin{align}
	\begin{aligned}
	\label{eqn::MCStudyTheoreticalQuantileAndES}
	Q_{\alpha}\left(Y | X \right) = X' (\gamma + z_\alpha \eta )   \qquad \text{and} \qquad \ES_{\alpha}\left(Y | X \right)  = X'(\gamma + \xi_\alpha \eta),
	\end{aligned}
\end{align}	
where $z_{\alpha}$ and $\xi_{\alpha}$ are the quantile and ES of the distribution $F(0,1)$, which implies that $\theta^q_{0} = \gamma + z_\alpha \eta$ and $\theta^e_{0} = \gamma + \xi_\alpha \eta$.
Furthermore, the conditional distributions of the quantile- and ES-residuals are given by
\begin{align} 
	\begin{aligned}
	\label{eqn::MCStudyDistributionResiduals}
	u^q | X  \sim F\left( -z_\alpha (X'\eta), \, (X'\eta)^2 \right) \qquad \text{and} \qquad u^e | X  \sim F \left( -\xi_\alpha (X'\eta), \, (X'\eta)^2 \right).
	\end{aligned}
\end{align}	

For the simulation study, we want to assess the performance of our regression procedure in various setups.
Thus, we specify $\gamma$, $\eta$ and $F$ in the following such that we get data which is homoscedastic (DGP-(1))  and heteroskedastic (DGP-(2)). Furthermore, we include a regression setup with multiple, correlated regressors and a leptocurtic conditional distribution (DGP-(3)),

\vspace{.5\baselineskip}

\begin{tabular}{llll}
	DGP-(1): & $X = (1,\, X_{2})$, & $X_{2} \sim \chi^2_1$ \quad and &  $Y|X \sim \mathcal{N}\bigl( -X_{2},\, 1 \bigr)$                                                                \\[.1cm]
	DGP-(2): & $X = (1,\, X_{2})$,  & $X_{2} \sim \chi^2_1$ \quad  and &  $Y|X \sim \mathcal{N} \bigl( -X_{2},\, (1 +  0.5 X_{2})^2 \bigr)$                                           \\[.1cm]
	DGP-(3): & $X  = (1,\, X_{2},\, X_{3})$ & \multicolumn{2}{l}{$X_{2},\, X_3 \sim U[0,1] \quad \text{with} \quad  \operatorname{corr}(X_2,X_3) = 0.5$ \quad and} \\
	         & \multicolumn{3}{l}{$Y|X              \sim t_5 \left( X_{2} - X_3, \, \left( 1 + X_{2} +  X_{3} \right)^2 \right)$.}
\end{tabular}
\\[.5\baselineskip]
We simulate all three processes 25,000 times with varying sample sizes of $n=250$, 500, 1000, 2000 and 5000 observations.
For each replication and for each of the sample sizes we regress the simulated $Y$'s on the covariates $X$ using our joint regression method for the probability level $\alpha=2.5\%$.

\subsection{Comparing the Specification Functions}
\label{sec::MCStudyChoiceOfTheGfunctions}

We start the discussion of the simulation results by investigating the numerical performance of the M-estimator based on different choices of the specification function\footnote{Following the reasoning of Section \ref{sec::ChoiceOfTheGfunctions} and \cite{Nolde2017,Ziegel2017}, we fix $G_1(z) = 0$ throughout the simulation study.} $\mathcal{G}_2$ used in the loss function in (\ref{eqn::regressionrhofunction}).
We use three natural examples resulting in positively homogeneous loss functions of order $b= -1$, $b= 0$ and $b= 0.5$ respectively\footnote{Our numerical simulations show that the numerical results are unaffected by different choices of the associated constants in (\ref{eqn::SpecFuncPosHomo1}) - (\ref{eqn::SpecFuncPosHomo3}).}, a bounded $G_2$ function and the (unbounded) exponential function:
\begin{align}	
\begin{aligned}
\label{eqn::G2Choices}	
\mathcal{G}_2(z) &= -1/z,  \qquad
\mathcal{G}_2(z) = - \log(-z),    \qquad
\mathcal{G}_2(z) = -\sqrt{-z},  \qquad \\
\mathcal{G}_2(z) &= \log \big( 1+\exp(z) \big),  \qquad \text{and} \qquad
\mathcal{G}_2(z) =\exp(z).
\end{aligned}
\end{align}

\Cref{fig::mc_mse} presents the sum (over the $2k$ regression parameters) of the mean squared errors (MSE) of the regression parameters for the three DGPs described above, different sample sizes and for the five choices of the specification functions given in (\ref{eqn::G2Choices}).
As implied by the asymptotic theory, we obtain consistent parameter estimates for all five choices of the specification functions as the MSEs converge to zero for all three DGPs.
However, they differ substantially with respect to their small sample properties.
The three positively homogeneous specifications result in the most accurate estimates, whereas the choices $\mathcal{G}_2(z) = -\sqrt{-z}$ and $\mathcal{G}_2(z) = - \log(-z)$ tend to perform slightly better than the choice  $\mathcal{G}_2(z) = -1/z$.
Furthermore, the bounded choice $\mathcal{G}_2(z) = \log\big(1 + \exp(z)\big)$ still performs better than the unbounded exponential function.

\begin{figure}[!htb]\centering
    \includegraphics{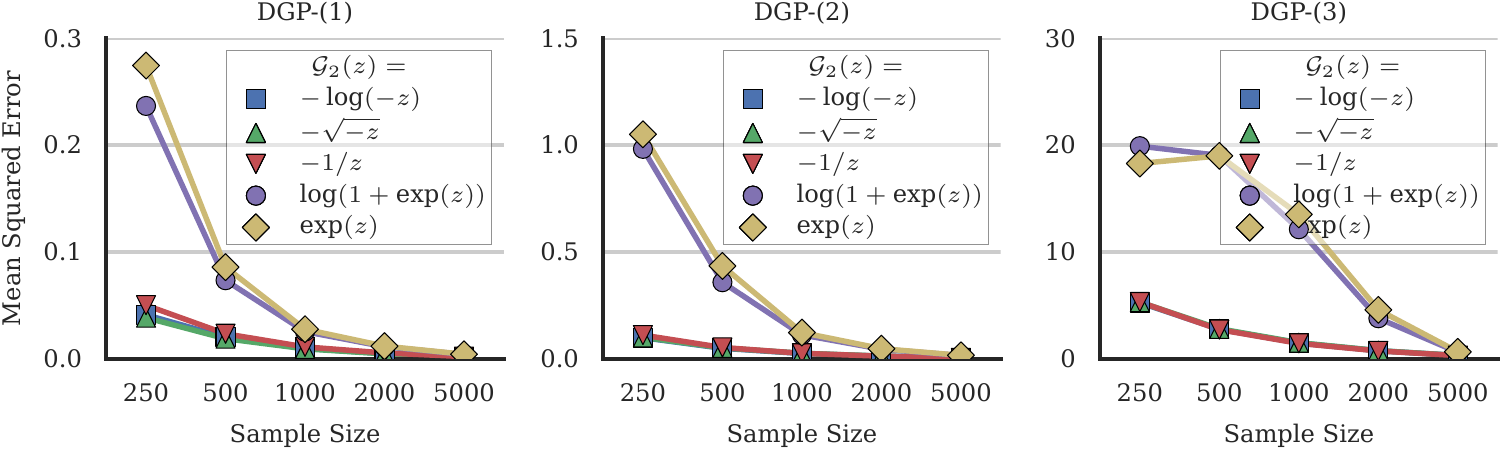}
	\caption{Sum of the mean squared errors of the parameter estimates for all three DGPs. The results are shown for the five choices of the specification functions given in (\ref{eqn::G2Choices}) and a range of sample sizes.}
	\label{fig::mc_mse}
\end{figure}

\Cref{tab::trueDiagValuesCovMatrices} reports the Frobenius norms of the lower triangular parts of the true asymptotic covariance matrices and of the respective (lower triangular) quantile-specific and the ES-specific sub-matrices for the three DGPs and for the five choices of the specification functions given in (\ref{eqn::G2Choices}).
For comparison, we also report the Frobenius norm of the lower triangular part of the asymptotic covariance of the quantile regression estimator.
We approximate the true asymptotic covariance matrix through Monte-Carlo integration with a sample size of $10^9$ using the formulas in \Cref{thm::AsymptoticNormalityPsi} and by using the true density and conditional truncated variance.
On average, the specification functions $\mathcal{G}_2(z) = -\log(-z)$ and $\mathcal{G}_2(z) = -\sqrt{-z}$ exhibit the smallest asymptotic covariances, closely followed by the third choice for a positively homogeneous loss function, $\mathcal{G}_2(z) = -1/z$.
The non-homogeneous choices lead to considerably larger asymptotic variances for all considered DGPs and sub-matrices.
Furthermore, by comparing the quantile-specific parameters of the joint estimation approach (from the positively homogeneous loss functions) to quantile regression estimates, we roughly obtain the same asymptotic efficiency.

\begin{table}[ht!]
	\caption{This table reports the Frobenius norms of the lower triangular parts of the asymptotic covariance matrices and the respective quantile-specific and the ES-specific sub-matrices for the three DGPs and for the five choices of the specification functions given in (\ref{eqn::G2Choices}).
	For comparison, we report the same quantity for the asymptotic covariance of the quantile regression estimator.}
	\label{tab::trueDiagValuesCovMatrices}
	\centering
	\footnotesize
	\begin{tabularx}{\linewidth}{X rrrr @{\hspace{0.5cm}} rrrr @{\hspace{0.5cm}} rrr}
		\toprule
		                                                         & \multicolumn{3}{c}{DGP-(1)} &  & \multicolumn{3}{c}{DGP-(2)} &  & \multicolumn{3}{c}{DGP-(3)} \\
		\cmidrule(lr){2-4}\cmidrule(lr){6-8}\cmidrule(lr){10-12} & Q    &   ES &          Full &  & Q    &   ES &          Full &  & Q      &     ES &      Full \\ \midrule
		$\mathcal{G}_2(z) = -\log(-z)$                           & 7.5  & 13.1 &           9.2 &  & 17.9 & 26.9 &          20.0 &  & 581.1  & 1739.1 &    1053.0 \\
		$\mathcal{G}_2(z) = -\sqrt{-z}$                          & 7.0  & 11.8 &           8.4 &  & 18.0 & 25.4 &          19.3 &  & 584.5  & 1740.1 &    1054.4 \\
		$\mathcal{G}_2(z) = -1/z$                                & 9.1  & 16.9 &          11.8 &  & 24.1 & 39.4 &          28.5 &  & 613.7  & 1851.9 &    1119.8 \\
		$\mathcal{G}_2(z) = \log(1 + \exp(z))$                   & 15.4 & 21.5 &          16.6 &  & 72.4 & 80.1 &          67.1 &  & 987.9  & 2393.0 &    1496.4 \\
		$\mathcal{G}_2(z) = \exp(z)$                             & 15.8 & 22.6 &          17.2 &  & 74.6 & 84.5 &          70.0 &  & 1001.9 & 2440.4 &    1524.6 \\
		Quantile Regression                                      & 6.8  &   -- &            -- &  & 21.4 &   -- &            -- &  & 600.5  &     -- &        -- \\ \bottomrule
	\end{tabularx}
\end{table}

\subsection{Comparing the Variance-Covariance Estimators}

In this section, we compare the empirical performance of the asymptotic covariance estimators discussed in \Cref{sec::CovEstimation}.
For the comparison of their precision, Figure \ref{fig::frobenius_norm} reports the average of the Frobenius norm of the lower triangular part of the differences between the estimated covariances and the empirical covariance of the estimated parameters.
\begin{figure}[htb]
	\includegraphics{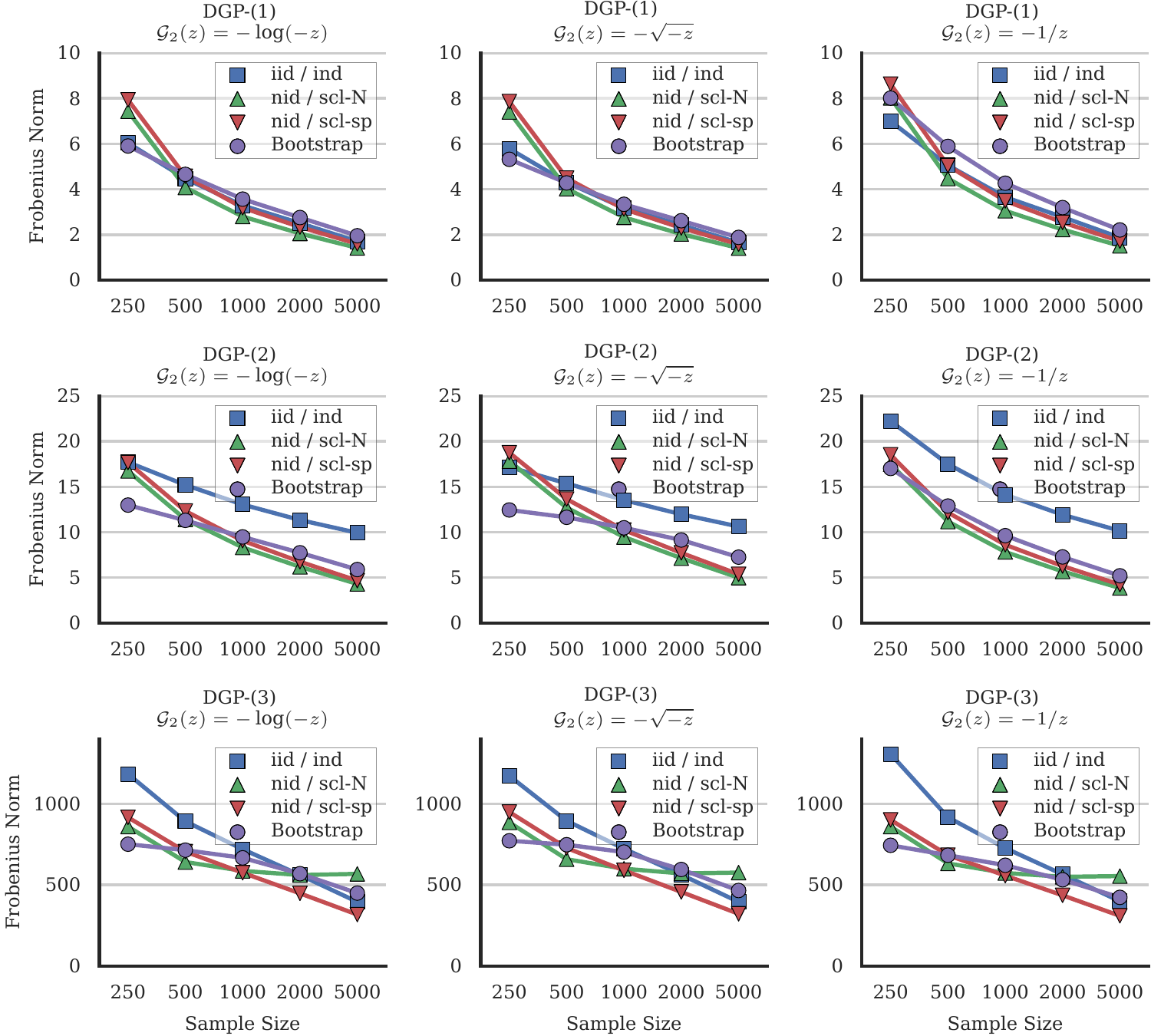}
	\caption{This figure compares four covariance estimation approaches described in Section \ref{sec::CovEstimation} for the three data generating processes, a range of sample sizes and the three positively homogeneous choices of the $\mathcal{G}_2$-functions.
	We report the average of the Frobenius norm of the lower triangular part of the differences between the estimated asymptotic covariances and the empirical covariance of the M-estimator.}
	\label{fig::frobenius_norm}
\end{figure}
We report results for the three homogeneous loss functions and the three DGPs, where each of the plots presents the average norm differences for the four covariance estimators (\textit{iid}/\textit{nid}, \textit{nid}/\textit{scl-N}, \textit{nid}/\textit{scl-sp} and the iid bootstrap) depending on the sample size.

We find that the \textit{iid}/\textit{nid} estimator performs well for the first, homoscedastic DGP whereas for the other two DGPs, it fails to capture the underlying more complicated dynamics of the data.
The \textit{nid}/\textit{scl-N} estimator outperforms the other estimation approaches in the first two DGPs, where the underlying conditional distribution follows a normal distribution whereas its performance drops for the third DGP, which follows a Student-$t$ distribution.
The performance of the flexible \textit{nid}/\textit{scl-sp} estimator is the most stable throughout all three DGPs.
Eventually, the bootstrap estimator accurately estimates the covariance for all three DGPs, whereas in comparison to the other estimators, it is particularly good in small samples.
The provided R package contains all four covariance estimators.

\section{Empirical Application}
\label{sec::EmpiricalApplication}

In this empirical application, we use our joint regression framework for forecasting the VaR and ES of the close-to-close log returns of the IBM stock.
For that purpose, we adopt the forecasting framework of \citet{Zikes2016} and jointly forecast the VaR and ES of daily financial returns $r_t$ by 
\begin{align}
	\label{eqn::ESRForecastingModel}
	Q_{\alpha}(r_t | \text{RV}_{t-1}) = \theta_1^q + \theta_2^q \text{RV}_{t-1} \quad \text{and} \quad
	\ES_{\alpha}(r_t | \text{RV}_{t-1}) = \theta_1^e + \theta_2^e \text{RV}_{t-1},
\end{align}
where $\text{RV}_t = (\sum_{i} r_{t,i}^2)^{1/2}$ denotes the realized volatility estimator \citep{Andersen1998} for day $t$, where $r_{t,i}$ denotes the $i$-th high-frequency return of day $t$.
Our dataset consists of the five minute returns of the IBM stock from January 3, 2001 to July 18, 2017 with total of 4120 days, which we obtain from the TAQ database.
We estimate the model parameters using a rolling window of 1000 days and evaluate the forecasts on the remaining 3120 days.

We compare the predictive power of this model against three standard models from the literature.
The first is the historical simulation (HS) approach, which forecasts the VaR and ES for day $t$ as the sample quantile and ES of the daily returns of the past 250 trading days.
The second is an AR(1)-GARCH(1,1)-$t$ model \citep{Bollerslev1986}, and the third is the Heterogeneous Auto-Regressive (HAR) model of \citet{Corsi2009}, based on the realized volatility estimates given above. 
Forecasts of the VaR and ES for the HAR model are obtained from the volatility forecasts and by assuming a Gaussian return distribution.
While the first two of these approaches rely on daily data only, the third one incorporates the same high frequency information as our approach.

We evaluate the forecasting power of the VaR and ES of these models by the class of strictly consistent loss (scoring) functions for the VaR and ES of \cite{Fissler2016}.
We use \textit{Murphy diagrams} introduced by \citet{Ehm2016} and \citet{Ziegel2017}, which provide a  parsimonious way to evaluate competing forecasts simultaneously for a full class of strictly consistent loss functions.
In fact, one forecasting model significantly dominates another one with respect to the full class of strictly consistent loss functions if and only if the elementary score differences plotted in the Murphy diagrams are strictly negative (positive). For further details on the theory and the implementation of Murphy diagrams, we refer to \citet{Ehm2016} and \citet{Ziegel2017}.

\begin{figure}[htb]\centering
	\includegraphics{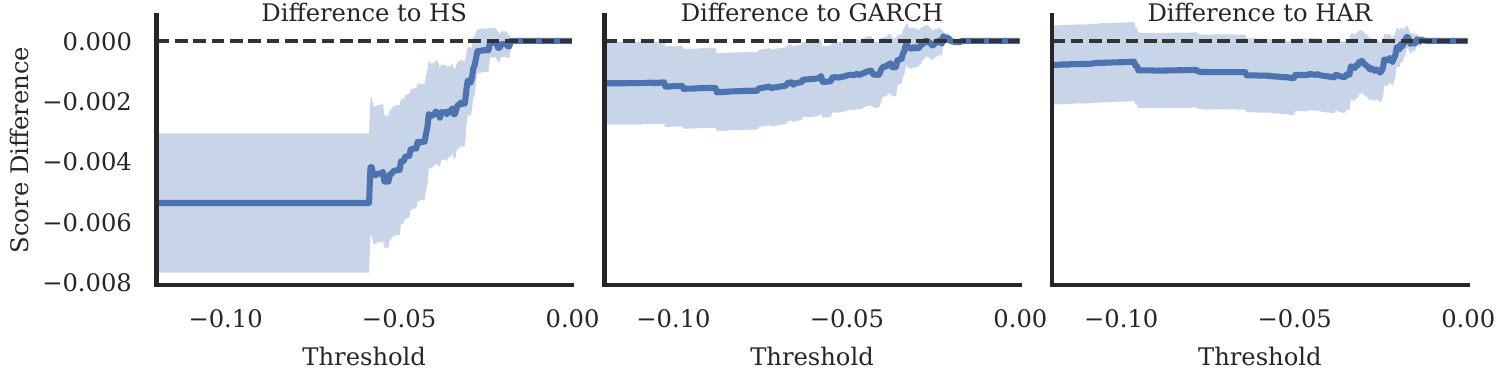}
	\caption{Elementary Score Differences of the VaR/ES Regression and the respective comparison models}
	\label{fig:murphy_diagram}
\end{figure}
\Cref{fig:murphy_diagram} displays the average of the elementary score differences of the joint VaR and ES regression model against the three alternative models together with the respective 95\% pointwise confidence bands for the elementary scores provided in \cite{Ziegel2017} for the pair VaR and ES.
Using this graphical method, we can see that the elementary score differences for the joint regression forecasting model against the historical simulation and AR(1)-GARCH(1,1)-$t$ model are significantly negative for the vast majority of threshold values.
This implies that the joint regression forecasting model significantly dominates these other two forecasting approaches.
Even though we also observe strictly negative elementary score differences in comparison against the HAR model, these differences are not significant and consequently, we cannot significantly outperform this model.

\section{Conclusion}
\label{sec::Conclusion}

In this paper, we introduce a joint regression technique for the quantile (the VaR) and the ES.
This regression approach relies on the class of strictly consistent joint loss functions introduced by \citet{Fissler2016}, which permits the joint elicitation of the quantile and the ES.
We introduce an M- and a Z-estimator for the parameters of the joint regression model.
Given a set of standard regularity conditions, we show consistency and asymptotic normality for both estimators, which we also verify numerically through extensive simulations.
The underlying loss functions, the estimating equations and the asymptotic covariance matrices of the estimators depend on the choice of two specification functions, which we investigate in terms of the resulting moment conditions, asymptotic efficiency, numerical performance and computation times.
In our numerical simulations, we find that choices resulting in positively homogeneous loss functions dominate other choices with respect to the aforementioned criteria.
Furthermore, we propose several estimation methods for the asymptotic covariance matrix, which are able to cope with different properties of the underlying data.
We provide an R package \citep[see][]{ESRegPackge}, which implements the M- and Z-estimation procedures where one can choose the underlying specification functions, the numerical optimization approach and the estimation method for the asymptotic covariance matrix.

Our new joint regression technique allows for a wide range of applications for the risk measures VaR and ES.
This regression approach can be used to model the ES (jointly with the VaR) by generalizing existing applications of quantile regression on VaR, such as e.g. in \cite{KoenkerXiao2006}, \cite{Engle2004}, \cite{Chernozhukov2001}, \cite{Zikes2016}, \cite{Halbleib2012}, \cite{Komunjer2013} and \cite{Xiao2015}.
As an illustration, we present an empirical application in this paper where we use this regression framework to jointly forecast VaR and ES based on realized volatility estimates.
Furthermore, \cite{BayerDimi2017} use this regression to develop an ES backtest which is particularly relevant in light of the recent introduction of ES into the Basel regulatory framework and the present lack of accurate backtesting methods for the ES.

\section*{Acknowledgements}

We thank Tobias Fissler, Lyudmila Grigoryeva, Roxana Halbleib, Phillip Heiler, Frederic Menninger, Winfried Pohlmeier, Patrick Schmidt, Johanna Ziegel and the participants of the Stochastics Colloquium on 11/30/2016 at the University of Konstanz for fruitful discussions and suggestions which inspired some of the results of this paper.
Financial support by the Heidelberg Academy of Sciences and Humanities (HAW) within the project ``Analyzing, Measuring and Forecasting Financial Risks by means of High-Frequency Data'', by the German Research Foundation (DFG) within the research group ``Robust Risk Measures in Real Time Settings'' and general support by the Graduate School of Decision Sciences (University of Konstanz) is gratefully acknowledged.
The computation in this work was performed on the computational resource bwUniCluster funded by the Ministry of Science, Research and the Arts Baden-Württemberg and the Universities of the State of Baden-Württemberg, Germany, within the framework program bwHPC.

\begin{appendices}

\section{Finite Moment Conditions}
\label{sec::GeneralMomentConditions}

For convenience of the supremum notation, for all $\theta \in \operatorname{int}(\Theta)$ and for $d>0$, we define the open neighborhood $U_d(\theta) = \{\tau \in \Theta :||\tau - \theta|| < d\}$ and its closure $\bar U_d(\theta) = \{\tau \in \Theta :||\tau - \theta|| \le d\}$.

\begin{enumerate}[label= ($\mathcal{M}$-\arabic*)]
	\item
	\label{MomCond::ConsistencyPsi}
	For Theorem \ref{thm::ConsistencyPsi}, we assume that the following moments are finite for some $d_0 > 0$:
	\begin{multicols}{2}
		\begin{itemize}[leftmargin=*]\itemsep0pt
			\item $\mathbb{E} [||X||^2 \sup_{\theta \in U_{d_0}(\theta_0)} | G_1^{(1)}(X' \theta^q)  | ]$
			\item $\mathbb{E} [||X||^2 \sup_{\theta \in U_{d_0}(\theta_0)} | G_1^{(2)}(X' \theta^q)  | ]$
			\item $\mathbb{E} [||X||^2 \sup_{\theta \in U_{d_0}(\theta_0)} | G_2(X' \theta^e)  |]$	
			\item $\mathbb{E} [||X||^3 \sup_{\theta \in U_{d_0}(\theta_0)} | G_2^{(1)}(X' \theta^e)  | ]$
			\item $\mathbb{E} [||X||^3 \sup_{\theta \in U_{d_0}(\theta_0)} | G_2^{(2)}(X' \theta^e)  | ]$
			\item $\mathbb{E} [||X||^2 \sup_{\theta \in U_{d_0}(\theta_0)} | G_2^{(1)}(X' \theta^e)  | \; \mathbb{E} [|Y| | X ] ]$
			\item $\mathbb{E} [||X||^2 \sup_{\theta \in U_{d_0}(\theta_0)} | G_2^{(2)}(X' \theta^e)  | \; \mathbb{E} [|Y| | X ] ]$
		\end{itemize}
	\end{multicols}

	\item
	\label{MomCond::ConsistencyRho}
	For Theorem \ref{thm::ConsistencyRho}, we assume that the following moments are finite:
	\begin{multicols}{2}
		\begin{itemize}[leftmargin=*]\itemsep0pt
			\item $\mathbb{E} [||X||^2 ]$
			\item $\mathbb{E} [\sup_{\theta \in \Theta} | G_1(X' \theta^q) | ]$
			\item $\mathbb{E} [ | G_1(Y) | ]$
			\item $\mathbb{E} [ | a(Y) | ]$
			\item $\mathbb{E} [||X|| \sup_{\theta \in \Theta} |G_2(X' \theta^e) | ]$
			\item $\mathbb{E} [\sup_{\theta \in \Theta} | G_2(X' \theta^e) | \; \mathbb{E} [|Y| | X ] ]$
			\item $\mathbb{E} [\sup_{\theta \in \Theta} | \mathcal{G}_2(X' \theta^e) | ]$
		\end{itemize}
	\end{multicols}

	\item
	\label{MomCond::AsymptoticNormalityPsi}
	For Theorem \ref{thm::AsymptoticNormalityPsi}, we assume that the following moments are finite for some constant $d_0 > 0$ and for all $\theta \in \bar U_{d_0}(\theta_0)$:
	\begin{itemize}[leftmargin=*]\itemsep0pt
		\item $\mathbb{E} [ ||X||^3 (\sup_{\tau \in \bar U_{d_0}(\theta_0)} G_1^{(1)}( X'\tau^q)) ( \sup_{\tilde \tau \in \bar U_{d_0}(\theta_0)} G_1^{(2)}( X' \tilde \tau^q))]$ 
		\item $\mathbb{E} [ ||X||^3 (\sup_{\tau \in \bar U_{d_0}(\theta_0)} G_1^{(1)}( X'\tau^q)) ( \sup_{\tilde \tau \in \bar U_{d_0}(\theta_0)} G_2^{(1)}( X' \tilde \tau^e))]$
		\item $\mathbb{E} [ ||X||^3 (\sup_{\tau \in \bar U_{d_0}(\theta_0)} G_2( X'\tau^e)) ( \sup_{\tilde \tau \in \bar U_{d_0}(\theta_0)} G_1^{(2)}( X' \tilde \tau^q)) ]$
		\item $\mathbb{E} [ ||X||^3 (\sup_{\tau \in \bar U_{d_0}(\theta_0)} G_2( X'\tau^e)) ( \sup_{\tilde \tau \in \bar U_{d_0}(\theta_0)} G_2^{(1)}( X' \tilde \tau^e)) ]$
		\item $\mathbb{E} [||X||^3 \sup_{\tau \in \bar U_{d_0}(\theta_0)}  ( G_1^{(1)}(X' \tau^q) )^2 ]$
		\item $\mathbb{E} [||X||^3 \sup_{\tau \in \bar U_{d_0}(\theta_0)}  (G_2(X' \tau^e))^2 ]$
		\item $\mathbb{E} [||X||^3 \sup_{\tau \in \bar U_{d_0}(\theta_0)}  G_1^{(1)}(X' \tau^q)  G_2(X' \tau^e) ]$
		\item $\mathbb{E} [ ||X||^5 (\sup_{\tau \in \bar U_{d_0}(\theta_0)} G_2^{(1)}( X'\tau^e)) ( \sup_{\tilde \tau \in \bar U_{d_0}(\theta_0)} G_2^{(2)}( X' \tilde \tau^e))]$
		\item $\mathbb{E} [ ||X||^5 (\sup_{\tau \in \bar U_{d_0}(\theta_0)} G_2^{(1)}( X'\tau^e))^2 ]$
		\item $\mathbb{E} [ ||X||^4 (\sup_{\tau \in \bar U_{d_0}(\theta_0)} G_2^{(1)}( X'\tau^e)) ( \sup_{\tilde \tau \in \bar U_{d_0}(\theta_0)} G_2^{(2)}( X' \tilde \tau^e)) \mathbb{E} [ |Y| | X ] ]$
		\item $\mathbb{E} [ ||X||^3 G_2^{(1)}(X' \theta^e) (\sup_{\tau \in \bar U_{d_0}(\theta_0)} G_2^{(1)}( X'\tau^e))  \mathbb{E} [ |Y| | X ] ]$
		\item $\mathbb{E} [ ||X||^3  G_2^{(1)}(X' \theta^e) (\sup_{\tau \in \bar U_{d_0}(\theta_0)} G_2^{(2)}( X'\tau^e)) \mathbb{E} [ Y^2 | X ] ]$
		\item $\mathbb{E} [ ||X||^3 (\sup_{\tau \in \bar U_{d_0}(\theta_0)} G_2^{(1)}( X'\tau^e)) ( \sup_{\tilde \tau \in \bar U_{d_0}(\theta_0)} G_2^{(2)}( X' \tilde \tau^e)) \mathbb{E}[ Y^2 | X ] ]$
	\end{itemize}

	\item
	\label{MomCond::AsymptoticNormalityRho}
	For Theorem \ref{thm::AsymptoticNormalityRho}, we assume that the following moments are finite for some constant $d_0 > 0$:
	\begin{multicols}{2}
		\begin{itemize}[leftmargin=*]\itemsep0pt
			\item $\mathbb{E} [ | G_1(Y) | ]$
			\item $\mathbb{E} [ | a(Y) | ]$
			\item $\mathbb{E} [ ||X|| \sup_{\theta \in \bar U_{d_0}(\theta_0)} | G_1^{(1)}(X' \theta^q) | ]$
			\item $\mathbb{E} [ ||X||^2 \sup_{\theta \in \bar U_{d_0}(\theta_0)} (G_1^{(1)}(X' \theta^q))^2 ]$	
			\item $\mathbb{E} [ ||X||^2 \sup_{\theta \in \bar U_{d_0}(\theta_0)} | G_1^{(1)}(X' \theta^q) G_2(X' \theta^e) | ]$
			\item $\mathbb{E} [ ||X|| \sup_{\theta \in \bar U_{d_0}(\theta_0)} | G_2(X' \theta^e) | ]$
			\item $\mathbb{E} [ ||X||^2 \sup_{\theta \in \bar U_{d_0}(\theta_0)} | G_2^{(1)}(X' \theta^e) | ]$
			\item $\mathbb{E} [ ||X||^2 \sup_{\theta \in \bar U_{d_0}(\theta_0)}  (G_2(X' \theta^e))^2 ]$
			\item $\mathbb{E} [ ||X||^4 \sup_{\theta \in \bar U_{d_0}(\theta_0)} (G_2^{(1)}(X' \theta^e))^2 ]$
			\item $ \mathbb{E} [ ||X|| \sup_{\theta \in \bar U_{d_0}(\theta_0)} | G_2^{(1)}(X' \theta^e) | \; \mathbb{E}[|Y| | X]]$
			\item $\mathbb{E} [ ||X||^3 \sup_{\theta \in \bar U_{d_0}(\theta_0)} (G_2^{(1)}(X' \theta^e))^2 \; \mathbb{E}[|Y| | X]]$
			\item $\mathbb{E} [ ||X||^2 \sup_{\theta \in \bar U_{d_0}(\theta_0)} (G_2^{(1)}(X' \theta^e))^2 \;  \mathbb{E}[Y^2|X]]$
		\end{itemize}
	\end{multicols}		
\end{enumerate}

\section{Proofs}
\label{sec::Proofs}

Henceforth, $||v||$ denotes the maximum norm for a vector $v \in \mathbb{R}^k$ and for a matrix $A$, $||A||$ denotes the row-sum matrix norm which is induced by the maximum norm for vectors.
For convenience of the supremum notation, for all $\theta \in \operatorname{int}(\Theta)$ and for some $d>0$, we define the open neighborhood $U_d(\theta) = \{\tau \in \Theta :||\tau - \theta|| < d\}$ and its closure $\bar U_d(\theta) = \{\tau \in \Theta :||\tau - \theta|| \le d\}$.
All references to Appendix \ref{sec::TechnicalResults} refer to the online supplement \cite{DimiBayer2017Supplement}.

\begin{proof}[Proof of Theorem \ref{thm::ConsistencyPsi}]
	
	We apply Theorem 2 from \cite{Huber1967} and show that the function $\psi(Y,X,\theta)$ as given in (\ref{eqn::regressionpsifunction}) satisfies the respective assumptions of this theorem.	
	Note that the parameter space $\Theta$ is assumed to be compact and thus, we do not have to show condition (B-4) in the notation of \cite{Huber1967}.
	As the product of continuous functions and the indicator function $\mathds{1}_{\{Y \le X' \theta^q \}}$, the function $\psi$ is measurable and regarded as a stochastic process in $\theta$, $\psi$ is separable in the sense of Doob as it is almost surely continuous in $\theta$ (\citealp{GikhmanSkorokhod2004}, p.164). 
	This condition assures measurability of the suprema\footnote{	
		Many other authors such as e.g. \cite{NeweyMcFadden1994,Andrews1994,VanderVaart1998} rely on outer probability in order to avoid these measurability issues.
	} given below and in Lemma \ref{lemma::normalityN3iicondition}.

	In oder to show that $\psi$ has a unique root at $\theta_0$, let us first define the sets
	\begin{align}
		U_\theta = \big\{ \omega \in \Omega \big| X(\omega)'\theta^q \not= X(\omega)'\theta^q_0 \big\}, \qquad \text{and} \qquad
		W_\theta = \big\{ \omega \in \Omega \big| X(\omega)'\theta^q = X(\omega)'\theta^q_0 \big\},
	\end{align} 
	for all $\theta \in \Theta$ such that $\Omega = W_\theta  \cup U_\theta $ and  $W_\theta  \cap U_\theta  = \emptyset$.
	We first show that $\mathbb{P}(U_\theta ) > 0$ for all $\theta \not= \theta_0$. In order to see this, we assume the converse, i.e. let us assume that for a fixed $\theta \not= \theta_0$, it holds that $\mathbb{P}(W_\theta ) = \mathbb{P} \big( X'\theta^q = X'\theta^q_0 \big) = 1$, which implies that
	\begin{align}
		(\theta^q - \theta^q_0)' \, \mathbb{E}[XX'] \, (\theta^q - \theta^q_0) = \mathbb{E} \big[ \big( X'\theta^q - X'\theta^q_0 \big)^2 \big] = 0.
	\end{align}
	However, since $\theta^q \not= \theta^q_0$, this contradicts the assumption that the matrix $\mathbb{E}[XX']$ is positive definite and we can conclude that $\mathbb{P}(U_\theta) > 0$.
	
	The quantity
	\begin{align*}
		\lambda_1(\theta) = \mathbb{E} \big[ \psi_1(Y,X,\theta) \big] =  1/\alpha \, \mathbb{E} \left[  X  \bigl( \alpha G_1^{(1)}(X' \theta^q) + G_2(X' \theta^e) \bigr) \big( F_{Y|X}(X' \theta^q) - F_{Y|X}(X' \theta^q_0) \big) \right]	
	\end{align*}	
	exists under the moment conditions \ref{MomCond::ConsistencyPsi} in Appendix \ref{sec::GeneralMomentConditions} and if $\theta^q = \theta^q_0$, it holds that $\lambda_1(\theta) = 0$.
	Now, we assume that $\theta \in \Theta$ such that $\theta^q \not= \theta^q_0$. 
	By splitting the expectation, we get that
	\begin{align*}
		\begin{aligned}
			& \lambda_1(\theta)'(\theta^q - \theta^q_0) \\
			= \; &1/\alpha \,  \mathbb{E} \left[ \bigl( \alpha G_1^{(1)}(X' \theta^q) + G_2(X' \theta^e) \bigr) \big( X'\theta^q - X' \theta^q_0 \big)  \big( F_{Y|X}(X' \theta^q) - F_{Y|X}(X' \theta^q_0) \big) \mathds{1}_{\{\omega \in W_\theta \}} \right] \\	
			+ \; &1/\alpha \, \mathbb{E} \left[ \bigl( \alpha G_1^{(1)}(X' \theta^q) + G_2(X' \theta^e) \bigr) \big( X'\theta^q - X' \theta^q_0 \big) \bigl( F_{Y|X}(X' \theta^q) - F_{Y|X}(X' \theta^q_0) \bigr) \mathds{1}_{\{\omega \in U_\theta \}} \right].	
		\end{aligned}
	\end{align*}
	The first summand is obviously zero since for all $\omega \in W_\theta $, $F_{Y|X}(X' \theta^q) - F_{Y|X}(X' \theta^q_0) = 0$.
	Since the distribution of $Y$ given $X$ has strictly positive density in a neighbourhood of $X' \theta^q_0$, we get that $F_{Y|X}$ is strictly increasing in a neighbourhood of $X' \theta^q_0$ and thus
	\begin{align}	
		\label{eqn::PositivityLambda1}
		\big( X'\theta^q - X'\theta^q_0 \big) \big( F_{Y|X}(X' \theta^q) - F_{Y|X}(X' \theta^q_0) \big) > 0
	\end{align}
	for all $\omega \in U_\theta$.
	Furthermore, since $\alpha G_1^{(1)}(X' \theta^q) + G_2(X' \theta^e) > 0$ for all $\theta \in \Theta$ and $\mathbb{P}(U_\theta) > 0$, we get that
	\begin{align*}
		&\lambda_1(\theta)'(\theta^q - \theta^q_0) \\
		&= 1/\alpha \, \mathbb{E} \left[ \bigl( \alpha G_1^{(1)}(X' \theta^q) + G_2(X' \theta^e) \bigr) \big( X'\theta^q - X' \theta^q_0 \big) \big( F_{Y|X}(X' \theta^q) - F_{Y|X}(X' \theta^q_0) \big) \mathds{1}_{\{\omega \in U_\theta \}} \right] > 0,
	\end{align*}
	and consequently $\lambda_1(\theta) \not= 0$.
	This implies that $\lambda_1(\theta) = 0$ if and only if $\theta^q = \theta^q_0$.
	Furthermore,
	\begin{align}
	 	\label{eqn::Lambda2}
		\lambda_2(\theta)
		= \mathbb{E} \left[ X G_2^{(1)}(X'\theta^e)  \left( X' \theta^q \, \bigl( F_{Y|X}(X' \theta^q)-\alpha \bigr)/\alpha + X' \theta^e - 1/\alpha \, \mathbb{E} \big[ Y \mathds{1}_{\{Y \le X' \theta^q\}} \big| X \big] \right) \right].					
	\end{align}
	Assuming that $\theta^q = \theta^q_0$, which results from $\lambda_1(\theta) = 0$, we get that $F_{Y|X}(X' \theta^q) = F_{Y|X}(X' \theta^q_0)  = \alpha$ and $1/\alpha \, \mathbb{E} \big[ Y \mathds{1}_{\{Y \le X' \theta^q_0\}} \big| X \big] = X' \theta^e_0$.
	Thus, (\ref{eqn::Lambda2}) simplifies to $\mathbb{E} \big[ (XX') G_2^{(1)} (X'\theta^e) \big] \big( \theta^e -  \theta^e_0 \big)$
	and by applying Lemma \ref{lemma::MatrixPositiveDefinite}, we get that the matrix $\mathbb{E} \big[ (XX')  G_2^{(1)}(X'\theta^e) \big]$ is positive definite for all $\theta \in \Theta$.
	Consequently, $\lambda_2(\theta) = 0$ if and only if $\theta^e = \theta^e_0$ and together with the arguments for $\lambda_1$, we get that $\lambda(\theta) = 0$ if and only if $\theta = \theta_0$.
	Eventually, assumption (B-2)' from Theorem 2 of \cite{Huber1967} follows directly from Lemma \ref{lemma::normalityN3iicondition}, which concludes this proof.
\end{proof}

\begin{proof}[Proof of Theorem \ref{thm::ConsistencyRho}]
	For this proof, we apply Theorem 5.7 from \cite{VanderVaart1998} and show that the respective assumptions of this theorem hold. 
	As in the proof of Theorem \ref{thm::AsymptoticNormalityPsi}, we can conclude measurability of the suprema since the process $\rho$ is continuous and consequently separable in the sense of Doob. 
	Thus, we do not have to rely on outer probability measures such as in \cite{VanderVaart1998}.
	We start by showing uniform convergence in probability of the empirical mean of the objective function by the help of Lemma 2.4 of \citet{NeweyMcFadden1994}. Since we have iid data, a compact parameter space $\Theta$ and $\rho(Y,X,\theta)$ is continuous for all $\theta \in \Theta$, it remains to show that there exists a dominating function  $ d(Y,X) \ge |\rho(Y,X,\theta)\big|$ for all $\theta \in \Theta$ with $\mathbb{E} \big[ d(Y,X) \big] < \infty$.
	We define
	\begin{align}
	\begin{aligned}
		d(Y,X) &=
		\sup_{\theta \in \Theta} \left|G_1(X' \theta^q) + 1/\alpha \, G_2(X' \theta^e) (X' \theta^q - Y)  \right|  + \big| G_1(Y) \big| \\
		&\quad+ \sup_{\theta \in \Theta}  \left| G_2(X' \theta^e) \big( X' \theta^e - X' \theta^q \big) \right| + \sup_{\theta \in \Theta}  \left| \mathcal{G}_2(X' \theta^e) \right| + \big| \alpha G_1(Y) +  a(Y) \big|
	\end{aligned}
	\end{align}		
	and it holds that $d(Y,X) \ge \big| \rho(Y,X,\theta) \big|$ for all $\theta \in \Theta$ and consequently, we can conclude uniform convergence in probability.
	
	We now show that $\mathbb{E} \big[\rho(Y,X,\theta) \big]$ has a unique and global minimum at $\theta = \theta_0$.
	For this, we assume that $\theta \in \Theta$ such that $\theta \not= \theta_0$ and we define the sets
	\begin{align}
		U_\theta &= \big\{ \omega \in \Omega \big| X(\omega)'\theta^q \not= X(\omega)'\theta^q_0 \quad \text{ or } \quad X(\omega)'\theta^e \not= X(\omega)'\theta^e_0 \big\} \quad \text{and} \\
		W_\theta &= \big\{ \omega \in \Omega \big|X(\omega)'\theta^q = X(\omega)'\theta^q_0 \quad \text{and} \quad X(\omega)'\theta^e = X(\omega)'\theta^e_0 \big\},
	\end{align} 
	such that $\Omega = U_\theta \cup W_\theta$ and $U_\theta \cap W_\theta = \emptyset$.
	We first show that $\mathbb{P}(U_\theta) > 0$ for all $\theta \not= \theta_0$. In order to see this, we assume the converse, i.e. we assume that $\mathbb{P}(W_\theta) = 1$, which implies that	$(\theta^q - \theta^q_0)' \, \mathbb{E}[XX'] \, (\theta^q - \theta^q_0) = \mathbb{E} \left[ \big( X'\theta^q - X'\theta^q_0 \big)^2 \right] = 0$, since  $\mathbb{P} \big( X'\theta^q = X\theta^q_0 \big) = 1$	and equivalently $(\theta^e - \theta^e_0)' \mathbb{E}[XX'] (\theta^e - \theta^e_0)  = 0$.
	However, since $\theta \not= \theta_0$ and consequently either $\theta^q \not= \theta^q_0$ or  $\theta^e \not= \theta^e_0$, this contradicts the assumption that the matrix $\mathbb{E}[XX']$ is positive definite and it follows that $\mathbb{P}(U_\theta) > 0$.

	From the joint elicitability property of the quantile and ES of \cite{Fissler2016}, Corollary 5.5 we get that for all $x \in \mathbb{R}^k$ such that $x' \theta^q \not= x'\theta_0^q$ or $x' \theta^e \not= x'\theta_0^e$, it holds that
	\begin{align}
		\mathbb{E} \big[ \rho(Y,X,\theta_0)\big| X = x \big]  < \mathbb{E} \big[\rho(Y,X,\theta) \big| X = x\big],
	\end{align}	
	since the distribution of $Y$ given $X$ has a finite first moment and a unique $\alpha$-quantile.
	Thus, for all $\omega \in U_\theta$,
	\begin{align}
		\label{eqn::ElicitabilityUsage}
		\mathbb{E} \big[ \rho(Y,X,\theta_0)\big| X  \big](\omega)  < \mathbb{E} \big[\rho(Y,X,\theta) \big| X \big](\omega).
	\end{align}	
	We now define the random variable
	\begin{align}
		h(X,\theta,\theta_0)(\omega) = \mathbb{E} \big[ \rho(Y,X,\theta_0) \big| X \big](\omega) - \mathbb{E} \big[ \rho(Y,X,\theta) \big| X \big](\omega),
	\end{align}
	and (\ref{eqn::ElicitabilityUsage}) implies that $h\big(X,\theta,\theta_0\big)(\omega) < 0$ for all $\omega \in U_\theta$.
	Since $\mathbb{P}(U_\theta) > 0$, this implies that $\mathbb{E} \left[ h(X,\theta,\theta_0) \mathds{1}_{\{\omega \in U_\theta\}} \right] < 0$.
	Furthermore, for all $\omega \in W_\theta$, it obviously holds that $h(X,\theta,\theta_0)(\omega) = 0$ and consequently $\mathbb{E} \left[ h(X,\theta,\theta_0) \mathds{1}_{\{\omega \in W_\theta\}} \right] = 0$.
	Thus, we get that
	\begin{align}
		\mathbb{E} \big[ h(X,\theta,\theta_0) \big] = \mathbb{E} \left[ h(X,\theta,\theta_0) \mathds{1}_{\{\omega \in U_\theta \}} \right] + \mathbb{E} \left[ h(X,\theta,\theta_0) \mathds{1}_{\{\omega \in W_\theta\}} \right] < 0
	\end{align}
	for all $\theta \in \Theta$ such that $\theta \not= \theta_0$, which shows that $\mathbb{E} \big[\rho(Y,X,\theta) \big]$ has a unique minimum at $\theta = \theta_0$.
\end{proof}

\begin{proof}[Proof of Theorem \ref{thm::AsymptoticNormalityPsi}]
We apply Theorem 3 of \cite{Huber1967} for the $\psi$-function as given in (\ref{eqn::regressionpsifunction}) and show the respective assumptions of this theorem.
Consistency of the Z-estimator is shown in Theorem \ref{thm::ConsistencyPsi}.
For the measureability and separability of the $\psi$ function, we refer to the proof of Theorem \ref{thm::ConsistencyPsi}. It is already shown in the proof of Theorem \ref{thm::ConsistencyPsi} that there exists a $\theta_0 \in \Theta$ such that $\lambda(\theta_0) = 0$.
For the technical conditions (N-3), we apply Lemma \ref{lemma::normalityN3icondition}, Lemma \ref{lemma::normalityN3iicondition} and Lemma \ref{lemma::normalityN3iiicondition}. 
It remains to show that $\mathbb{E} \big[ ||\psi(Y,X,\theta_0)||^2 \big] < \infty$, which follows from the subsequent computation of $C$ and the Moment Conditions \ref{MomCond::AsymptoticNormalityPsi} in Appendix \ref{sec::GeneralMomentConditions}.
The asymptotic covariance matrix is given by $\Lambda^{-1} C \Lambda^{-1}$, where $C = \mathbb{E} \big[ \psi(Y,X,\theta_0) \, \psi(Y,X,\theta_0)' \big]$
and
\begin{align}
\Lambda = \left. \frac{\partial \lambda(\theta)}{\partial \theta} \right|_{\theta = \theta_0} = \begin{pmatrix} \Lambda_{11} & \Lambda_{12} \\ \Lambda_{21} & \Lambda_{22}  \end{pmatrix} = \begin{pmatrix} \left. \frac{\partial \lambda_1(\theta)}{\partial \theta^q} \right|_{\theta_0} & \left. \frac{\partial \lambda_1(\theta)}{\partial \theta^e} \right|_{\theta_0} \\ \left. \frac{\partial \lambda_2(\theta)}{\partial \theta^q} \right|_{\theta_0} & \left. \frac{\partial \lambda_2(\theta)}{\partial \theta^e} \right|_{\theta_0} \end{pmatrix}.
\end{align}
Straightforward calculations yield the matrix $C$ as given in (\ref{eqn::AsyCovMatrixC11}) - (\ref{eqn::AsyCovMatrixC22}).
For the computation of $\Lambda$, we first notice that the function
\begin{align} 
	\mathbb{E} \big[ \psi(Y,X,\theta)\big| X \big] =
	\begin{pmatrix}
		\frac{1}{\alpha} \big(F_{Y|X}(X'\theta^q) -\alpha \big) \bigl( \alpha G_1^{(1)}(X' \theta^q) + G_2(X' \theta^e) \bigr) \vspace{0.05cm} \\
		X G_2^{(1)}(X'\theta^e) \left( X' \theta^e - X' \theta^q + \frac{1}{\alpha} \mathbb{E} \big[ (X' \theta^q - Y) \mathds{1}_{\{Y \le X' \theta^q\}} \big| X \big] \right)
	\end{pmatrix}
\end{align}	
is continuously differentiable for all $\theta$ in some neighborhood $U_{d}(\theta_0)$ around $\theta_0$, since the distribution $F_{Y|X}$ has a density which is strictly positive, continuous and bounded in this area.
Let us choose a value $\tilde \theta \in U_d(\theta_0)$ such that $X'\tilde \theta \le X'\theta$.
Then,
\begin{align}
	\begin{aligned}
	\label{eqn::ESDifferentiation}
	\frac{\partial}{\partial \theta^q} \mathbb{E} \big[Y \mathds{1}_{\{Y \le X' \theta^q\}} \big| X \big]
	&= \frac{\partial}{\partial \theta^q} \mathbb{E} \big[Y \mathds{1}_{\{Y \le X' \tilde \theta^q\}} \big| X \big] + \frac{\partial}{\partial \theta^q} \mathbb{E} \big[ Y \mathds{1}_{\{X' \tilde \theta^q < Y \le X'\theta^q\}} \big| X \big] \\
	&= \frac{\partial}{\partial \theta^q} \int_{X'\tilde \theta^q}^{X'\theta^q} y   f_{Y|X}(y) \mathrm{d}y
	= X (X' \theta^q) f_{Y|X}(X'\theta^q).
	\end{aligned}
\end{align}	
We consequently get that for all $\theta \in U_d(\theta_0)$,
\begin{align*}
	\frac{\partial}{\partial \theta^q} \mathbb{E} \big[ \psi_1(Y,X,\theta) \big| X  \big] &= 1/\alpha \, (XX') \left[ \bigl( \alpha G_1^{(1)}(X' \theta^q) + G_2(X' \theta^e) \bigr) f_{Y|X} (X' \theta^q) \right. \\
	&\qquad\qquad\qquad \left. + G_1^{(2)}(X'\theta^q) \big(F_{Y|X} (X' \theta^q)-\alpha \big) \right], \\
	\frac{\partial}{\partial \theta^e} \mathbb{E} \big[ \psi_1(Y,X,\theta) \big| X \big] &= \frac{\partial}{\partial \theta^q} \mathbb{E} \big[ \psi_2(Y,X,\theta) \big| X \big] = 1/\alpha \,  (XX') G_2^{(1)}(X'\theta^e) \big( F_{Y|X}(X' \theta^q)-\alpha \big), \\
	\frac{\partial}{\partial \theta^e} \mathbb{E} \big[ \psi_2(Y,X,\theta) \big| X \big] &= 1/\alpha \,
	(XX') G_2^{(2)}(X'\theta^e)  \left[ X' \theta^q \big( F_{Y|X}(X' \theta^q)-\alpha \big) + \alpha (X' \theta^e) -  \mathbb{E} \big[ Y \mathds{1}_{\{Y \le X' \theta^q\}} \big| X \big] \right] \\
	& \qquad \quad  + (XX') G_2^{(1)}(X'\theta^e).
\end{align*}
In order to conclude that $	\frac{\partial}{\partial \theta} \mathbb{E} \big[ \mathbb{E} \big[ \psi(Y,X,\theta)\big| X \big]\big] = \mathbb{E} \left[ \frac{\partial}{\partial \theta} \mathbb{E}\big[ \psi(Y,X,\theta) \big| X \big] \right]$,
we apply a measure-theoretical version of the Leibniz integration rule, which requires that the derivative of the integrand exists and is absolutely bounded by some integrable function $d(Y,X)$, independent of $\theta$.
For the first term, this can easily be obtained by defining
\begin{align*}
	d(Y,X) = \sup_{\theta \in U_{d}(\theta_0)} &\left| \left|  1/\alpha \, (XX') \left[ \big( \alpha G_1^{(1)}(X' \theta^q) + G_2(X' \theta^e) \big) f_{Y|X} (X' \theta^q) + G_1^{(2)}(X' \theta^q) \big( F_{Y|X} (X' \theta^q)-\alpha \big) \right] \right| \right|,
\end{align*} 
which has finite expectation by the Moment Conditions \ref{MomCond::AsymptoticNormalityPsi}. 
The other two terms follow the same reasoning. 
Inserting $\theta = \theta_0$ eventually shows (\ref{eqn::Lambda11}) and (\ref{eqn::Lambda22}).
\end{proof}

\begin{proof}[Proof of Theorem \ref{thm::AsymptoticNormalityRho}]
	For this proof, we apply Theorem 5.23 from \cite{VanderVaart1998} and show that the respective assumptions of this theorem hold.
	Theorem \ref{thm::ConsistencyRho} shows consistency of the M-estimator.
	The map $(Y,X) \mapsto \rho(Y,X,\theta)$ is obviously measurable as the sum of measurable functions.
	Furthermore, the map $\theta \mapsto \rho(Y,X,\theta)$ is almost surely differentiable since the only point of non-differentiability occurs where $Y = X'\theta^q$, which is a nullset with respect to the joint distribution of $Y$ and $X$ and for all $\theta \in \Theta$ such that  $Y \not= X'\theta^q$, its derivative is given by $\psi(Y,X,\theta)$.
	Local Lipschitz continuity with square-integrable Lipschitz-constant follows from Lemma \ref{lemma::RhoLipschitzContinuity}.	
	We have already seen in the proof of Theorem \ref{thm::ConsistencyRho} that the function $\mathbb{E} \big[ \rho(Y,X,\theta) \big]$ is uniquely minimized at the point $\theta_0$ and is twice continuously differentiable and consequently admits a second-order Taylor expansion at $\theta_0$. Thus, we have shown the necessary assumptions of Theorem 5.23 from \cite{VanderVaart1998}. 
	
	For the computation of the covariance matrix, we notice that the distribution of $Y$ given $X$ has a density $f_{Y|X}$ in a neighborhood of $X'\theta_0$, which is strictly positive, continuous and bounded. 
	Therefore, by the same arguments as in (\ref{eqn::ESDifferentiation}), we get that $\frac{\partial}{\partial \theta^q} \mathbb{E} \big[ G_1(Y) \mathds{1}_{\{Y \le X' \theta^q\}} \big| X \big] = X G_1(X'\theta^q) f_{Y|X}(X'\theta^q)$.
	Thus, straight-forward calculations yield that for all $\theta \in U_{d}(\theta_0)$, it holds that
	$\frac{\partial}{\partial \theta} \mathbb{E} \big[ \rho(Y,X,\theta) \big| X \big] = \mathbb{E} \big[ \psi(Y,X,\theta) \big| X \big]$
	and by applying the Leibniz integration rule such as in the proof of Theorem \ref{thm::AsymptoticNormalityPsi}, we finally get that
	\begin{align}
		\frac{\partial}{\partial \theta} \mathbb{E} \big[ \rho(Y,X,\theta)\big] = \mathbb{E} \big[ \psi(Y,X,\theta)\big].
	\end{align} 
	Consequently, the asymptotic covariance matrix equals the one given in Theorem \ref{thm::AsymptoticNormalityPsi}.
\end{proof}

\section{Technical Results}
\label{sec::TechnicalResults}

\begin{lemma} 
	\label{lemma::normalityN3iicondition}
	Let
	\begin{align}
	u(Y,X,\theta,d) = \sup_{\tau \in \bar U_d(\theta)} \big|\big|\psi(Y,X,\tau) - \psi(Y,X,\theta) \big|\big|
	\end{align}
	and assume that Assumption \ref{ass::Model}, Assumption \ref{ass::GeneralAssumptionsConsistency} and the Moment Conditions \ref{MomCond::ConsistencyPsi} in Appendix \ref{sec::GeneralMomentConditions} hold.
	Then, there are strictly positive real numbers $b$ and $d_0$, such that
	\begin{align} \label{eqn::normalityEu1condition}
	\mathbb{E} \big[ u(Y,X,\theta,d) \big] &\le b \cdot d  \quad\;\, \mbox{for} \quad ||\theta - \theta_0|| + d \le d_0, 
	\end{align}
	and for all $d \ge 0$.
\end{lemma}

\begin{proof}[Proof of Lemma \ref{lemma::normalityN3iicondition}]
	
	For measurability of the suprema, we refer to the proof of Theorem \ref{thm::ConsistencyPsi}.
	Let in the following $d > 0$ and $\theta \in \Theta$ such that $||\theta - \theta_0|| + d \le d_0$.
	We first notice that for some fixed $X \in \mathbb{R}^k$ and for all $\tau \in \bar U_d(\theta)$, it holds that
	\begin{align}
	\left| \mathds{1}_{\{Y \le X' \theta^q\}} - \mathds{1}_{\{ Y \le X' \tau^q\}} \right| \le  \mathds{1}_{\{ X' \theta^q_{-} \le Y \le X'\theta^q_{+} \}}
	\end{align}	
	for all $Y \in \mathbb{R}$ and for some $\theta^q_{-}, \theta^q_{+} \in \bar U_d(\theta)$. Since $ \bar U_d(\theta)$ is compact, we get that
	\begin{align} \label{eqn::supindicatorfunction}
	\sup_{\tau \in \bar U_d(\theta)} \left| \mathds{1}_{\{Y \le X' \theta^q\}} - \mathds{1}_{\{ Y \le X' \tau^q\}} \right| \le \mathds{1}_{\{ X' \theta^q_{-} \le Y \le X'\theta^q_{+} \}}
	\end{align}	
	for all $Y \in \mathbb{R}$ and for some values $\theta^q_{-}, \theta^q_{+} \in \bar U_d(\theta)$.
	Note that the values $\theta^q_{-}$ and $\theta^q_{+}$ depend on $X$ and $\theta$, however they are independent of $Y$.
	Consequently, it holds that
	\begin{align} 
	\begin{aligned}
	\label{eqn::IndicatorFunctionInequalityOd}
	&\mathbb{E} \left[ \left. \sup_{\tau \in \bar U_d(\theta)} \left| \mathds{1}_{\{Y \le X' \theta^q\}} - \mathds{1}_{\{ Y \le X' \tau^q\}} \right| \right| X \right] 
	\le \mathbb{E} \left[ \left. \mathds{1}_{\{ X' \theta^q_{-} \le Y \le X'\theta^q_{+} \}} \right| X \right] \\
	= \; & F_{Y|X} \big(X'\theta^q_{+} \big) - F_{Y|X} \big(X'\theta^q_{-} \big)
	=  f_{Y|X}(X' \tilde \theta^q) \big(X'\theta^q_{+} - X'\theta^q_{-} \big) \\
	\le \; & 2 ||X|| \cdot  \sup_{ \tau \in \bar U_d(\theta)}f_{Y|X}(X' \tau^q) \cdot d,
	\end{aligned}
	\end{align}
	where we apply the mean value theorem for some $\tilde \theta^q$ on the line between $\theta^q_{-}$ and $\theta^q_{+}$, i.e. $\tilde \theta^q \in \bar U_d(\theta)$.

	For the first component of $\psi$, we get that
	\begin{align}
	\begin{aligned}
	\label{eqn::Psi1Od}
	&\mathbb{E} \left[  \sup_{\tau \in \bar U_d(\theta)} \big|\big| \psi_1(Y,X,\theta) - \psi_1(Y,X,\tau) \big|\big| \right] \\
	\le \; &\mathbb{E} \left[ \sup_{\tau \in \bar U_d(\theta)}  \left| \left| X \left( G_1^{(1)}(X' \theta^q) - G_1^{(1)}(X' \tau^q) + \frac{G_2(X' \theta^e) - G_2(X' \tau^e)}{\alpha} \right) \right| \right| \right] \\
	&\qquad + \mathbb{E} \left[ \sup_{\tau \in \bar U_d(\theta)}  \left| \left| X \left(G_1^{(1)}(X' \tau^q) + \frac{G_2(X' \tau^e)}{\alpha} \right) \right| \right| \cdot  \mathbb{E} \left[ \left. \sup_{\tau \in \bar U_d(\theta)}  \left| \mathds{1}_{\{Y \le X' \theta^q\}} - \mathds{1}_{\{Y \le X' \tau^q\}} \right| \right| X \right] \right].
	\end{aligned}
	\end{align}
	The first term in (\ref{eqn::Psi1Od}) is $\mathcal{O}(d)$ since $G_1^{(1)}(X' \theta^q)$ and $G_2(X' \theta^e)$ are continuously differentiable functions w.r.t $\theta$ and thus, by the mean value theorem we get that
	\begin{align}
	\begin{aligned}
	\label{eqn::continuousfunctionOd}
	\sup_{\tau \in \bar U_d(\theta)} \big| G_1^{(1)}(X' \theta^q) - G_1^{(1)}(X' \tau^q) \big| &\le \sup_{\tilde \tau \in \bar U_d(\theta)} \big|\big| X G_1^{(2)}(X' \tilde \tau^q) \big| \big| \cdot \sup_{\tau \in \bar U_d(\theta)} \big|\big| \theta^q - \tau^q \big| \big| \\
	&\le \sup_{\tilde \tau \in \bar U_d(\theta)} \big|\big| X G_1^{(2)}(X' \tilde \tau^q) \big| \big| \cdot d,
	\end{aligned}
	\end{align}
    and the respective moments are finite by assumption.
	The same arguments hold for the function $G_2$.	
	For the second term  in (\ref{eqn::Psi1Od}), we apply (\ref{eqn::IndicatorFunctionInequalityOd}) and thus get that
	\begin{align}
	\begin{aligned}
	&\mathbb{E} \left[ \sup_{\tau \in \bar U_d(\theta)}  \left| \left| X \left(G_1^{(1)}(X' \tau^q) + \frac{G_2(X' \tau^e)}{\alpha} \right) \right| \right| \cdot  \mathbb{E} \left[ \left. \sup_{\tau \in \bar U_d(\theta)}  \left| \mathds{1}_{\{Y \le X' \theta^q\}} - \mathds{1}_{\{Y \le X' \tau^q\}} \right| \right| X \right] \right] \\	
	\le \; &\mathbb{E} \left[ \sup_{\tau \in \bar U_d(\theta)} \left| \left| X \left( G_1^{(1)}(X' \tau^q) + \frac{G_2(X' \tau^e)}{\alpha} \right)\right|\right|  ||X|| \cdot \sup_{\tau \in \bar U_d(\theta)} f_{Y|X}(X' \tau^q)  \right] \cdot d.	
	\end{aligned}	
	\end{align}
	Since the density $f_{Y|X}$ is bounded in a neighborhood of $X'\theta^q_0$ and the respective moments are finite by assumption, we get that this term is also $\mathcal{O}(d)$.

	For the second component of $\psi$, we get that
	\begin{align*}
	& \mathbb{E} \left[ \sup_{\tau \in \bar U_d(\theta)}  \big|\big| \psi_2(Y,X,\theta) - \psi_2(Y,X,\tau) \big|\big| \right] \\
    \le \; & \mathbb{E} \left[\sup_{\tau \in \bar U_d(\theta)}  \big| \big| X (X' \theta^e - X' \theta^q) G_2^{(1)}(X' \theta^e) - X (X' \tau^e - X' \tau^q) G_2^{(1)}(X' \tau^e) \big| \big| \right]                                                                                                            \\
	& \qquad + \mathbb{E} \left[  \left|\left| \frac{X G_2^{(1)}(X' \theta^e) X' \theta^q }{\alpha} \right|\right| \cdot \mathbb{E} \left[ \left. \sup_{\tau \in \bar U_d(\theta)}  \big|  \left(\mathds{1}_{\{Y \le X' \theta^q\}} - \mathds{1}_{\{Y \le X' \tau^q\}} \right) \big| \right| X \right] \right] \\
	& \qquad  + \mathbb{E} \left[ \mathbb{E}\left[ \left. \sup_{\tau \in \bar U_d(\theta)}  \left| \left| \mathds{1}_{\{Y \le X' \tau^q\}} \left( \frac{X G_2^{(1)}(X' \theta^e) X' \theta^q}{\alpha} - \frac{X  G_2^{(1)}(X' \tau^e) X' \tau^q}{\alpha} \right) \right| \right| \; \right| X  \right]\right]  \\
	& \qquad + \mathbb{E} \left[ \left|\left| \frac{X G_2^{(1)}(X' \theta^e)}{\alpha} \right| \right| \cdot  \mathbb{E} \left[ \left. \sup_{\tau \in \bar U_d(\theta)}  \left| Y \left(\mathds{1}_{\{Y \le X' \theta^q\}} - \mathds{1}_{\{Y \le X' \tau^q\}} \right) \right| \right| X \right] \right] \\
	& \qquad   + \mathbb{E} \left[\mathbb{E} \left[ \left. \sup_{\tau \in \bar U_d(\theta)}  \left|\left| \frac{Y \mathds{1}_{\{Y \le X' \tau^q\}}}{\alpha} \big( X G_2^{(1)}(X' \theta^e) - X G_2^{(1)}(X' \tau^e) \big) \right| \right| \right| X  \right] \right] \\
	= \; & (\mathrm{i}) + (\mathrm{ii}) + (\mathrm{iii}) + (\mathrm{iv}) + (\mathrm{v}).
	\end{align*}
	
	The first, third and fifth term are linearly bounded by (\ref{eqn::continuousfunctionOd}) since the functions $(X' \theta^e - X' \theta^q) G_2^{(1)}(X' \theta^e)$ and $(X' \theta^q) G_2^{(1)} (X' \theta^e)$ and $G_2^{(1)}(X' \theta^e)$ are continuously differentiable.
	For the second term, we use the arguments from (\ref{eqn::IndicatorFunctionInequalityOd}).
	For the fourth term, we use similar arguments as in (\ref{eqn::IndicatorFunctionInequalityOd}), and get that there exist some $\theta^q_{-},\theta^q_{+} \in \bar U_d(\theta)$ and a value $\tilde \theta^q$ on the line between $\theta^q_{-}$ and $\theta^q_{+}$, such that
	\begin{align}
	\begin{aligned}
	\label{eqn::IndicatorFunctionPsi2FourthTerm}
	&\mathbb{E} \left[ \left|\left| \frac{X G_2^{(1)}(X' \theta^e)}{\alpha} \right| \right|  \mathbb{E} \left[ \left.  \sup_{\tau \in \bar U_d(\theta)} \left| Y \left(\mathds{1}_{\{Y \le X' \theta^q\}} - \mathds{1}_{\{Y \le X' \tau^q\}} \right) \right| \right| X \right] \right] \\
	\le \; &\mathbb{E} \left[\left|\left| \frac{X G_2^{(1)}(X' \theta^e)}{\alpha} \right| \right| \mathbb{E} \left[ |Y| \left. \mathds{1}_{\{ X' \theta^q_{-} \le Y \le X'\theta^q_{+} \}} \right| X \right] \right] \\
	= \; &\mathbb{E} \left[ \left|\left| \frac{X G_2^{(1)}(X' \theta^e)}{\alpha} \right| \right| \int_{X'\theta^q_{-}}^{X'\theta^q_{+}} |y| f_{Y|X} (y) \mathrm{d}y \right] \\
	\le \; &\mathbb{E} \left[\left|\left| \frac{X G_2^{(1)}(X' \theta^e)}{\alpha} \right| \right| |X' \tilde \theta^q| f_{Y|X}(X' \tilde \theta^q) \big(X'\theta^q_{+} - X'\theta^q_{-} \big) \right] \\
	\le \; &\frac{2}{\alpha} \mathbb{E} \left[G_2^{(1)}(X' \theta^e) \big|\big| X \big|\big|^2 \sup_{\tau \in \bar U_d(\theta)} |X' \tau^q| f_{Y|X}(X' \tau^q)  \right] \cdot d
	= \mathcal{O}(d)
	\end{aligned}
	\end{align}
	since $f_{Y|X}$ is bounded in a neighborhood of $X'\theta_0$ and the respective moments exist by assumption.
	This concludes the proof of the lemma.
\end{proof}

\begin{lemma} \label{lemma::MatrixPositiveDefinite}
	Let the random variable $X \in \mathbb{R}^k$ with distribution $\mathbb{P}$ be such that its second moments exist and the matrix $\mathbb{E}[X X']$ is positive definite. Furthermore, let $\tilde \Theta \subset \mathbb{R}^k$ be a compact subspace with nonempty interior and let $g: \mathbb{R}^k \times \tilde \Theta \to \mathbb{R}$ be a strictly positive function.
	Then, the matrix
	\begin{align}
	\mathbb{E}\big[ (X X') g(X,\theta) \big] 
	\end{align}
	is also positive definite.
\end{lemma}

\begin{proof}[Proof of Lemma \ref{lemma::MatrixPositiveDefinite}]
	Since  $\mathbb{E}[X X']$ is positive definite, we know that for all $z \in \mathbb{R}^k$ with $z \not= 0$, it holds that $0 < z' \mathbb{E}[X X'] z = \mathbb{E}[z' (X X') z] = \mathbb{E}[(X'z)^2]$ and consequently  $\mathbb{P} \big( X'z \not= 0 \big) > 0 $.
	Since $\sqrt{g(X,\theta)}$ is a strictly positive scalar for all $\theta \in \tilde \Theta$, it also holds that $\mathbb{P} \big( (X'z) \sqrt{g(X,\theta)} \not= 0 \big) > 0$ and thus, for all $z \not= 0$,
	\begin{align}
	z' \mathbb{E}\big[(X X') g(X,\theta) ] z = \mathbb{E}\left[ \left(X'z \sqrt{g(X,\theta)} \right)^2 \right] > 0.
	\end{align}
	This positivity statement holds since $\big(X'z \sqrt{g(X,\theta)}\big)^2$ is a non-negative random variable and $\mathbb{P} \big( (X'z) \sqrt{g(X,\theta)} \not= 0 \big) > 0$.
	This shows that the matrix $\mathbb{E}\big[ (X X') g(X,\theta) \big]$ is positive definite. 
\end{proof}

\begin{lemma} \label{lemma::normalityN3icondition}			
	Assume that Assumption \ref{ass::Model}, Assumption  \ref{ass::GeneralAssumptionsConsistency} and the Moment Conditions \ref{MomCond::AsymptoticNormalityPsi} in Appendix \ref{sec::GeneralMomentConditions} hold.
	Then, for
	\begin{align}
	\lambda(\theta) = \mathbb{E}\big[ \psi(Y,X,\theta)\big],
	\end{align}
	there are strictly positive numbers $a,d_0$, such that
	\begin{align}
	||\lambda(\theta)|| \ge a \cdot ||\theta-\theta_0||  \quad\;\; \mbox{for} \quad ||\theta - \theta_0||  \le d_0.
	\end{align}
\end{lemma}

\begin{proof}[Proof of Lemma \ref{lemma::normalityN3icondition}]
	Let $d_0 > 0$ and let $||\theta - \theta_0||  \le d_0$. Then, applying the mean value theorem, we get that
	\begin{align}
	\lambda_1(\theta) =  \frac{1}{\alpha} \mathbb{E} \left[  (XX') \big( \alpha G_1^{(1)}(X' \theta^q) + G_2(X' \theta^e) \big)  f_{Y|X}(X' \tilde \theta^q) \right]	 (\theta^q - \theta^q_0)
	\end{align}
	for some  $\tilde \theta^q$ on the line between $\theta^q$ and $\theta^q_0$.
	Similarly, for the second component we get that
	\begin{align}
	\begin{aligned}
	\label{eqn::Lambda2InverseLipschitz}
	\lambda_2(\theta)
	= \; & \mathbb{E} \left[X  \frac{G_2^{(1)}(X' \theta^e) f_{Y|X}(X' \tilde \theta^q)}{\alpha} \big[ X' (\theta^q - \theta^q_0) \big] \big[X' (\tilde{\theta^q} - \theta^q) \big] \right] \\
	+\;&\mathbb{E} \big[ (X X') G_2^{(1)}(X' \theta^e) \big]	 (\theta^e - \theta^e_0),
	\end{aligned}
	\end{align}
	where $\tilde \theta^q$ lies on the line between $\theta^q$ and $\theta^q_0$.

	We first assume that $||\theta - \theta_0|| = ||\theta^q  - \theta^q_0||$, i.e. $||\theta^q  - \theta^q_0|| \ge ||\theta^e  - \theta^e_0||$. 
	Since the matrix 
	\begin{align}
	A(\theta) := \mathbb{E} \left[  (XX') \frac{ \big( \alpha G_1^{(1)}(X' \theta^q) + G_2(X' \theta^e) \big)}{\alpha} f_{Y|X}(X' \tilde \theta^q) \right]
	\end{align}
	exists and has full rank for all $\theta \in \Theta$ by Lemma \ref{lemma::MatrixPositiveDefinite} and is obviously symmetric, $A$ has strictly positive real Eigenvalues $\gamma_1(\theta),\dots,\gamma_k(\theta)$ with minimum $\gamma_{(1)}(\theta)$ and we thus get that\footnote{For a symmetric matrix $A$ with full rank, we can find an orthogonal basis of Eigenvectors $\{v_1, \dots, v_k \}$ with corresponding nonzero Eigenvalues $\{ \gamma_1(\theta) , \dots , \gamma_k(\theta) \}$ such that $x = \sum b_j v_j$ with $b_j \in \mathbb{R}$. Then, $||Ax|| = || A \sum b_j v_j || = || \sum b_j A v_j || = || \sum b_j \gamma_j v_j || \ge \min{|\gamma_j|} \cdot || \sum b_j v_j || = \min{|\gamma_j|} \cdot ||x||$.}
	\begin{align} \label{eqn::MatrixVectorNormGeEq}
	||\lambda(\theta)|| \ge ||\lambda_1(\theta)|| &= || A(\theta) (\theta^q  - \theta^q_0) || \ge \gamma_{(1)}(\theta) \cdot || \theta^q  - \theta^q_0 || \\
	&\ge \left( \inf_{||\theta - \theta_0 || \le d_0} \gamma_{(1)}(\theta) \right) \cdot || \theta^q  - \theta^q_0 || = c_1 \, || \theta  - \theta_{0} ||.
	\end{align}
	Since $||\theta - \theta_0 || \le d_0$ is a compact set and the function $\theta \mapsto \inf_{||\theta - \theta_0 || \le d_0} \gamma_{(1)}(\theta)$, where $\gamma_{(1)}(\theta)$ is the smallest Eigenvalue of the matrix $A(\theta)$, is continuous\footnote{
		This follows since the entries of the matrix $A(\theta)$ are continuous in $\theta$ as the expectation of a continuous function which is dominated by an integrable function is again continuous by the dominated convergence theorem. Furthermore, the Eigenvalues of a matrix are the solution of the characteristic polynomial, which has continuous coefficients since our matrix entries are continuous in $\theta$. Eventually, since the roots of any polynomial with continuous coefficients are again continuous, we can conclude that the Eigenvalues of $A(\theta)$ are continuous in $\theta$.
	},
	we get that the infimum coincides with the minimum and thus, the constant $c_1 := \inf_{||\theta - \theta_0 || \le d_0} \gamma_{(1)}(\theta)$ is strictly positive and does not depend on $\theta$.
	\par 
	Now, we assume that $||\theta - \theta_0|| = ||\theta^e  - \theta^e_0|| \le d_0$, i.e. $||\theta^e  - \theta^e_0|| \ge ||\theta^q  - \theta^q_0||$. 
	For the first term of $\lambda_2(\theta)$, given in (\ref{eqn::Lambda2InverseLipschitz}), we define the vector
	\begin{align}
	b(\theta) :=  \mathbb{E} \left[X \frac{G_2^{(1)}(X' \theta^e) f_{Y|X}(X' \tilde \theta^q)}{\alpha} \big[ X' (\theta^q - \theta^q_0) \big] \big[X' \tilde\theta^q- X'\theta^q) \big] \right],
	\end{align}
	and for its $l$-th component, we get that
	\begin{align} 
	\begin{aligned}
	\label{eqn::BiComponentOd02}
	|b_l(\theta)| &= \left| \sum_{i,j}  (\theta^q_i - \theta^q_{0i}) (\tilde \theta^q_j - \theta^q_j)  \mathbb{E} \left[X_i X_j X_l \frac{G_2^{(1)}(X' \theta^e) f_{Y|X}(X' \tilde \theta^q)}{\alpha} \right] \right| \\
	&\le  \sum_{i,j}  \mathbb{E} \left[ \left| X_i X_j X_l \frac{G_2^{(1)}(X' \theta^e) f_{Y|X}(X' \tilde \theta^q)}{\alpha} \right| \right] \cdot |\theta^q_i - \theta^q_{0i}| \cdot |\tilde \theta^q_j - \theta^q_j|  \\
	&\le c_2 \sum_{i,j}  |\theta^q_i - \theta^q_{0i}| \cdot |\tilde \theta^q_j - \theta^q_j| \\
	&\le c_2 k^2 ||\theta - \theta_0||^2,
	\end{aligned}
	\end{align} 	
	for all $l = 1,\dots,k$, which implies that
	\begin{align}
	\label{eqn::BvectorOd02}
	||b(\theta)|| \le c_3 || \theta - \theta_0||^2,
	\end{align}
	for some $c_3 > 0$.
	For $D(\theta) := \mathbb{E} \big[ (XX') G_2^{(1)}(X' \theta^e) \big]$, it holds that $||D(\theta) (\theta^e - \theta^e_0)|| \ge c_4 ||\theta^e - \theta^e_0|| = c_4 ||\theta - \theta_0||$ for $c_4 > 0$ by the same arguments as in (\ref{eqn::MatrixVectorNormGeEq}). 
	From (\ref{eqn::BiComponentOd02}), we can choose $d_0$ small enough such that  
	\begin{align}
	\label{eqn::Choiced_0SmallEnough}
	2 ||b(\theta)|| \le 2 c_3 ||\theta - \theta_0||^2 \le c_4 ||\theta - \theta_0|| \le ||D(\theta) (\theta^e - \theta^e_0)||.
	\end{align}
	Furthermore, by the submultiplicativity of the matrix norm, we also get that $||D(\theta) (\theta^e - \theta^e_0)|| \le ||D(\theta)|| \cdot ||\theta^e - \theta^e_0|| = c_5 ||\theta^e - \theta^e_0|| $ and by the inverse triangle inequality, we get that 
	\begin{align}
	||\lambda(\theta) || \ge ||\lambda_2(\theta) || = \big|\big| D(\theta) (\theta^e - \theta^e_0) + b(\theta) \big|\big| \ge \Big| ||D(\theta) (\theta^e - \theta^e_0)|| - ||b(\theta)|| \Big|.
	\end{align}
	From (\ref{eqn::Choiced_0SmallEnough}), we can choose $d_0$ small enough such that  $||D (\theta^e - \theta^e_0)|| > 2 ||b||$ and thus
	\begin{align}
	\Big| ||D (\theta^e - \theta^e_0)|| - ||b|| \Big| &= ||D (\theta^e - \theta^e_0)|| - ||b|| \ge \frac{1}{2} ||D (\theta^e - \theta^e_0)|| \\
	&\ge \frac{c_4}{2} \, || \theta^e - \theta^e_0 || = \frac{c_4}{2} \, || \theta - \theta_0 ||.
	\end{align}
\end{proof}

\begin{lemma} \label{lemma::normalityN3iiicondition}
	Let
	\begin{align}
	u(Y,X,\theta,d) = \sup_{\tau \in \bar U_d(\theta)} \big|\big|\psi(Y,X,\tau) - \psi(Y,X,\theta) \big|\big|.
	\end{align}
	and assume that Assumption \ref{ass::Model}, Assumption \ref{ass::GeneralAssumptionsConsistency} and the Moment Conditions \ref{MomCond::AsymptoticNormalityPsi} in Appendix \ref{sec::GeneralMomentConditions} hold.
	Then, there are strictly positive numbers $c$ and $d_0$, such that
	\begin{align}\label{eqn::normalityEu2condition}
	\mathbb{E}\big[ u(Y,X,\theta,d)^2 \big] &\le c \cdot d  \quad\;\, \mbox{for} \quad ||\theta - \theta_0|| + d \le d_0,
	\end{align}
	and for all $d \ge 0$.
\end{lemma}	

\begin{proof}[Proof of Lemma \ref{lemma::normalityN3iiicondition}]
	Let in the following $d > 0$ and $\theta \in \Theta$ such that $||\theta - \theta_0|| + d \le d_0$.
	It holds that
	\begin{align}
	\left( \sup_{\tau \in \bar U_d(\theta)} \big|\big|\psi(Y,X,\tau) - \psi(Y,X,\theta) \big|\big| \right)^2 
	= \sup_{\tau \in \bar U_d(\theta)} \big|\big|\psi(Y,X,\tau) - \psi(Y,X,\theta) \big|\big|^2
	\end{align}
	and consequently, we show that 
	\begin{align}
	\mathbb{E} \left[ \sup_{\tau \in \bar U_d(\theta)} \big|\big|\psi_j(Y,X,\tau) - \psi_j(Y,X,\theta) \big|\big|^2 \right] = \mathcal{O}(d)
	\end{align}
	for both components $j = 1,2$ and for some $d >0$ small enough.

	For the first squared component, we get that
	\begin{align*}
	&\mathbb{E} \left[ \sup_{\tau \in \bar U_{d}(\theta)} \big|\big| \psi_1(Y,X,\tau) - \psi_1(Y,X,\theta) \big|\big|^2 \right]\\
	\le \;&\; \max \left(\left| \frac{1-\alpha}{\alpha} \right|^2, 1\right) \cdot  \mathbb{E} \left[ \sup_{\tau \in \bar U_{d}(\theta)}  \left| \left| X \big( \alpha G_1^{(1)}(X' \theta^q) + G_2(X' \theta^e) - \alpha  G_1^{(1)}(X' \tau^q) -  G_2(X' \tau^e) \big) \right| \right|^2   \right]  \\
	&\qquad  + \frac{2}{\alpha^2} \mathbb{E} \left[ \sup_{\tau \in \bar U_{d}(\theta)}  \left| \left| X \big( \alpha  G_1^{(1)}(X' \tau^q) + G_2(X' \tau^e) \big) \right| \right|^2  ||X|| \sup_{\tau \in \bar U_{d}(\theta)}  f_{Y|X}(X' \tau^q)\; \right] \cdot d\\
	&\qquad  + \frac{2}{\alpha^2} \max \big(1-\alpha, \alpha \big) \mathbb{E} \Biggl[ \sup_{\tau \in \bar U_{d}(\theta)} \big| \big| X \big( \alpha G_1^{(1)}(X' \theta^q) + G_2(X' \theta^e) - \alpha  G_1^{(1)}(X' \tau^q) -  G_2(X' \tau^e) \big) \big| \big|  \\
	&\qquad \qquad \qquad \cdot  \left| \left| X \big( \alpha  G_1^{(1)}(X' \tau^q) + G_2(X' \tau^e) \big) \right| \right| \Biggr], 			
	\end{align*}
	where we apply (\ref{eqn::IndicatorFunctionInequalityOd}) for the second summand.
	The remaining two summands can be  bounded linearly by the arguments given in (\ref{eqn::continuousfunctionOd}) since $G_1^{(1)}$ and $G_2$ are continuously differentiable functions and the respective moments are finite.
	\par 
	For the second component of $\psi$, we get that
	\begin{align}
	\begin{aligned} \label{eqn::psi2SquaredFiveComponents}
	&\big|\big| \psi_2(Y,X,\tau) - \psi_2(Y,X,\theta) \big|\big| \\
	&\quad \le \big|\big| X (X' \theta^e - X' \theta^q) G_2^{(1)}(X' \theta^e) - X (X' \tau^e -  X' \tau^q) G_2^{(1)}( X' \tau^e)\big|\big| \\
	&\qquad + \left| \left| \frac{X G_2^{(1)}(X' \theta^e) X' \theta^q }{\alpha} \left(\mathds{1}_{\{Y \le X' \theta^q\}} - \mathds{1}_{\{Y \le X' \tau^q\}} \right) \right| \right| \\
	&\qquad +\left| \left| \mathds{1}_{\{Y \le X' \tau^q\}} \left( \frac{X G_2^{(1)}(X' \theta^e) X' \theta^q }{\alpha} - \frac{X G_2^{(1)}(X' \tau^e) X' \tau^q }{\alpha}\right) \right| \right|\\
	&\qquad +\left| \left| \frac{X G_2^{(1)}(X' \theta^e)}{\alpha} Y \left(\mathds{1}_{\{Y \le X' \theta^q\}} - \mathds{1}_{\{Y \le X' \tau^q\}} \right)\right| \right| \\
	&\qquad + \left| \left| \frac{Y \mathds{1}_{\{Y \le X' \tau^q\}}}{\alpha} \big(X G_2^{(1)}(X' \theta^e) - X G_2^{(1)}(X' \tau^e)\big)\right| \right| \\
	&\quad = (\mathrm{i}) + (\mathrm{ii}) + (\mathrm{iii}) + (\mathrm{iv}) + (\mathrm{v}).
	\end{aligned}
	\end{align}
	Thus, in order to evaluate $\mathbb{E} \left[ \sup_{\tau \in \bar U_{d}(\theta)} \big|\big| \psi_2(Y,X,\tau) - \psi_2(Y,X,\theta) \big|\big|^2 \right]$, we have to consider all the cross products out of the five summands in (\ref{eqn::psi2SquaredFiveComponents}). 
	Since the techniques applied are very similar, we only show details for two of the cross products.
	\begin{align*}
	&\mathbb{E} \left[ \sup_{\tau \in \bar U_{d}(\theta)} \, (\mathrm{ii}) \cdot (\mathrm{v}) \right] \\
	= \; &\mathbb{E}\left[ \sup_{\tau \in \bar U_{d}(\theta)}\left| \left| \frac{X G_2^{(1)}(X' \theta^e) X' \theta^q }{\alpha} \left(\mathds{1}_{\{Y \le X' \theta^q\}} - \mathds{1}_{\{Y \le X' \tau^q\}} \right) \right| \right| \right. \\
	&\qquad \quad \cdot \left. \left| \left| \frac{Y \mathds{1}_{\{Y \le X' \tau^q\}}}{\alpha} \big(X G_2^{(1)}(X' \theta^e) - X G_2^{(1)}(X' \tau^e)\big)\right| \right| \right] \\		
	\le \; & \frac{1}{\alpha^2} \mathbb{E} \left[  \left| \left|X G_2^{(1)}(X' \theta^e) X' \theta^q   \right| \right| \cdot \mathbb{E} \big[|Y| \big| X \big] \cdot ||X|| \cdot \sup_{\tau \in \bar U_{d}(\theta)} \big| \big| G_2^{(1)}(X' \theta^e) - G_2^{(1)}(X' \tau^e) \big| \big| \right] \\			
	\le \; &\frac{1}{\alpha^2} \mathbb{E} \left[  \left| \left| X G_2^{(1)}(X' \theta^e) X' \theta^q  \right| \right| \cdot \mathbb{E} \big[|Y| \big| X \big] \cdot ||X|| \cdot \sup_{\tau \in \bar U_{d}(\theta)} \big| \big| X G_2^{(2)}(X' \tau^e) \big| \big| \right] \cdot d \\							
	= \; &\mathcal{O}(d),
	\end{align*}	
	by (\ref{eqn::continuousfunctionOd}) since $G_2^{(1)}$ is continuously differentiable.

	The following crossproducts can be bounded analogously by bounding the indicator functions and by applying the mean value theorem as in (\ref{eqn::continuousfunctionOd}):
	$(\mathrm{i})^2$,
	$(\mathrm{iii})^2$,
	$(\mathrm{v})^2$,
	$(\mathrm{i}) \cdot (\mathrm{iii})$,
	$(\mathrm{i}) \cdot (\mathrm{iv})$,
	$(\mathrm{i}) \cdot (\mathrm{v})$,
	$(\mathrm{ii}) \cdot (\mathrm{iv})$,
	$(\mathrm{ii}) \cdot (\mathrm{v})$,
	$(\mathrm{iii}) \cdot (\mathrm{iv})$,
	$(\mathrm{iii}) \cdot (\mathrm{v})$ and
	$(\mathrm{iv}) \cdot (\mathrm{v})$.
	\par 
	A second type of technique, similar to the arguments in (\ref{eqn::IndicatorFunctionPsi2FourthTerm}) arises in the cases
	$(\mathrm{ii})^2$,
	$(\mathrm{iv})^2$ and
	$(\mathrm{ii}) \cdot (\mathrm{iv})$. 
	We get that there exists $\theta^q_{-},\theta^q_{+} \in \bar U_d(\theta)$ and a value $\tilde \theta^q$ on the line between $\theta^q_{-}$ and $\theta^q_{+}$, such that
	\begin{align*}
		\mathbb{E} \left[ \sup_{\tau \in \bar U_{d}(\theta)} (\mathrm{iv})^2 \right]
		&\le \mathbb{E} \left[ \left| \left| \frac{X G_2^{(1)}(X' \theta^e)}{\alpha} \right| \right|^2 \mathbb{E} \left[ \left. \sup_{\tau \in \bar U_{d}(\theta)} \big| Y \left(\mathds{1}_{\{Y \le X' \theta^q\}} - \mathds{1}_{\{Y \le X' \tau^q\}} \right) \big|^2 \right| X \right] \right] \\
		&\le \mathbb{E} \left[\left| \left| \frac{X G_2^{(1)}(X' \theta^e)}{\alpha} \right| \right|^2 \mathbb{E} \left[ \left. Y^2 \, \mathds{1}_{\{ X' \theta^q_{-} \le Y \le X'\theta^q_{+} \}} \right| X \right] \right] \\
		&= \mathbb{E} \left[ \left| \left| \frac{X G_2^{(1)}(X' \theta^e)}{\alpha} \right| \right|^2 \int_{X' \theta^q_{-}}^{ X' \theta^q_{+}} y^2 f_{Y|X} (y) \mathrm{d}y \right] \\
		&\le \mathbb{E} \left[\left| \left| \frac{X G_2^{(1)}(X' \theta^e)}{\alpha} \right| \right|^2 (X' \tilde \theta^q)^2  f_{Y|X}(X' \tilde \theta^q) \big(X'\theta^q_{+} - X'\theta^q_{-} \big) \right] \\
		&\le \frac{2}{\alpha} \mathbb{E} \left[ \big|\big|X\big|\big|^3  \big(G_2^{(1)}(X' \theta^e)\big)^2 \cdot \sup_{\tau \in \bar U_d(\theta)} (X' \tau^q)^2  f_{Y|X}(X' \tau^q) \right] \cdot d \\ 
		&= \mathcal{O}(d),		
	\end{align*}					
	where we apply a multivariate version of the mean value theorem and notice that $f_{Y|X}$ is bounded.
\end{proof}

\begin{lemma} \label{lemma::RhoLipschitzContinuity}
	Assume that Assumption \ref{ass::Model}, Assumption \ref{ass::GeneralAssumptionsConsistency} and the Moment Conditions \ref{MomCond::AsymptoticNormalityRho} in Appendix \ref{sec::GeneralMomentConditions} hold.
	Then, the function $\rho(Y,X,\theta)$, given in (\ref{eqn::regressionrhofunction}) is locally Lipschitz continuous in $\theta$ in the sense that for all $\theta_1,\theta_2 \in U_{d}(\theta_0)$ in some neighborhood of $\theta_0$, it holds that
	\begin{align}
		\big| \rho(Y,X,\theta_1) - \rho(Y,X,\theta_2) \big| \le K(Y,X) \cdot \big|\big| \theta_1 - \theta_2 \big|\big|,
	\end{align}
	where $\mathbb{E} \big[K(Y,X)^2 \big] < \infty$.
\end{lemma}

\begin{proof}
	We start the proof by splitting the $\rho$ function into two parts,
	\begin{align} 
	\rho(Y,X,\theta) = \rho_1(Y,X,\theta) + \rho_2(Y,X,\theta),
	\end{align}	
	where
	\begin{align} 
	\rho_1(Y,X,\theta) &= \mathds{1}_{\{Y \le X' \theta^q\}} \left(G_1(X'\theta^q) - G_1(Y) + \frac{1}{\alpha} G_2(X'\theta^e) (X' \theta^q - Y) \right), \\
	\rho_2(Y,X,\theta) &= G_2(X' \theta^e) \big( X' \theta^e - X' \theta^q \big) - \mathcal{G}_2(X' \theta^e) - \alpha G_1(X'\theta^q) + a(Y).
	\end{align}	
	Local Lipschitz continuity of $\rho_2$ follows since it is a continuously differentiable function and thus locally Lipschitz.
	We consequently get that for some $d > 0$ and for all $\theta_1,\theta_2 \in U_{d}(\theta_0)$, it holds that
	\begin{align}
	\big| \rho_2(Y,X,\theta_1) - \rho_2(Y,X,\theta_2) \big| &\le \big|\big| \theta_1 - \theta_2 \big|\big|
	\cdot \sup_{\theta \in U_{d}(\theta_0)} \left|\left| \begin{pmatrix}
	-X G_2(X' \theta^e) - \alpha X G_1^{(1)}(X'\theta^q)  \\
	X G_2^{(1)}(X' \theta^e) \big( X' \theta^e - X' \theta^q \big)
	\end{pmatrix} \right|\right|,
	\end{align}
	with Lipschitz-constant
	\begin{align}
	K(Y,X) = \sup_{\theta \in U_{d}(\theta_0)} \left|\left| \begin{pmatrix}
	-X G_2(X' \theta^e) - \alpha X G_1^{(1)}(X'\theta^q)  \\
	X G_2^{(1)}(X' \theta^e) \big( X' \theta^e - X' \theta^q \big)
	\end{pmatrix} \right|\right|,
	\end{align} 
	which is square-integrable by the moment conditions \ref{MomCond::AsymptoticNormalityRho}.
	
	For the function $\rho_1$, we consider three cases.
	First, let $\theta_1, \theta_2 \in \Theta$ such that $X'\theta_1^q \le X'\theta_2^q < Y$. Then it holds that,
	\begin{align}
	\rho_1(Y,X,\theta_1) = \rho_1(Y,X,\theta_2) = 0,
	\end{align}
	since $\mathds{1}_{\{Y \le X' \theta^q_1\}} =  \mathds{1}_{\{Y \le X' \theta^q_2 \}} = 0$, which is obviously a Lipschitz continuous function.
	\par
	Second, let $\theta_1, \theta_2 \in \Theta$ such that $Y \le X'\theta_1^q \le X'\theta_2^q$. Then, for $\theta = \theta_1, \theta_2$,
	\begin{align}
	\rho_1(Y,X,\theta) = G_1(X'\theta^q) -  G_1(Y) + \frac{1}{\alpha} G_2(X'\theta^e) (X' \theta^q - Y),
	\end{align}
	which is a continuously differentiable function and thus
	\begin{align}
	\big| \rho_1(Y,X,\theta_1) - \rho_1(Y,X,\theta_2) \big| &\le \big|\big| \theta_1 - \theta_2 \big|\big|
	\cdot \sup_{\theta \in U_{d}(\theta_0)} \left|\left| \begin{pmatrix}
	X G_1^{(1)}(X'\theta^q) + \frac{1}{\alpha} X G_2(X'\theta^e) \\
	\frac{1}{\alpha} X G_2^{(1)}(X'\theta^e) (X' \theta^q - Y)
	\end{pmatrix} \right|\right|.
	\end{align}
	\par 
	Finally, let $\theta_1, \theta_2 \in \Theta$ such that $X'\theta_1^q < Y \le X'\theta_2^q$. Then, since $G_1$ is increasing, we get that
	\begin{align*}
	\big| \rho_1(Y,X,\theta_1) - \rho_1(Y,X,\theta_2) \big| 
	&= \left|G_1(X'\theta^q_2) -  G_1(Y) + \frac{1}{\alpha} G_2(X'\theta^e_2) (X' \theta^q_2 - Y)\right| \\
	&\le \big| G_1(X'\theta^q_2) - G_1(X'\theta^q_1) \big| + \left| \frac{1}{\alpha} G_2(X'\theta^e_2) (X' \theta^q_2 - X' \theta^q_1)\right| \\
	&\le \big|\big| \theta^q_1 - \theta^q_2 \big|\big| \cdot  \sup_{\theta \in U_{d}(\theta_0)}  \left( \big|\big| X G_1^{(1)}(X'\theta^q)\big|\big| + \frac{1}{\alpha} \big|\big| X  G_2(X'\theta^e) \big|\big| \right). 
	\end{align*}
	Thus, the function $\rho(Y,X,\theta)$ is locally Lipschitz continuous in $\theta$ with square-integrable Lipschitz constants, $\mathbb{E}\big[ K(Y,X)^2 \big] < \infty$ by the Moment Conditions \ref{MomCond::AsymptoticNormalityRho} in Appendix \ref{sec::GeneralMomentConditions}.
\end{proof}

\begin{proposition}
	\label{prop::SampleVarianceEqualsOurVariance}
	Let $Y$ be a real-valued random variable with distribution function $F$, finite first and second moments and a unique $\alpha$-quantile $q_\alpha = F^{-1}(\alpha)$.
	Then,
	\begin{align}
	\frac{1}{\alpha^2} \int_{-\infty}^{q_\alpha} \int_{-\infty}^{q_\alpha} F(x \wedge y) - F(x) F(y) \mathrm{d}x \mathrm{d}y = \frac{1}{\alpha} \operatorname{Var}(Y|Y\le q_\alpha) + \frac{1-\alpha}{\alpha} \big(q_\alpha - \xi_\alpha \big)^2,
	\end{align}
	where $\xi_\alpha = \mathbb{E} \big[ Y \big| Y \le q_\alpha \big]$ denotes the $\alpha$-ES of $Y$. 
\end{proposition}

\begin{proof}
	We first notice that for a distribution $F$ with finite second moment und unique $\alpha$-quantile, it holds that
	\begin{align}
		\label{eqn::ESestimationProofEXformula}
		\mathbb{E} \big[ Y \big| Y \le q_\alpha \big] &= - \frac{1}{\alpha}  \int_{-\infty}^{q_\alpha}  F(x) \mathrm{d}x + q_\alpha  \qquad \text{ and } \\
		\label{eqn::ESestimationProofEX2formula}
		\mathbb{E} \big[ Y^2 \big| Y   \le q_\alpha \big] &= - \frac{2}{\alpha}  \int_{-\infty}^{q_\alpha} x F(x) \mathrm{d}x + q_\alpha^2,
	\end{align}
	which can be obtained by using the identity
	\begin{align}
		Y \mathds{1}_{\{ Y \le q_\alpha \}} = \mathds{1}_{\{ Y \le q_\alpha \}} \left( \int_0^\infty \mathds{1}_{\{ Y > t \}} \, \mathrm{d}t - \int_{-\infty}^0 \mathds{1}_{\{ Y \le t \}} \, \mathrm{d}t \right)
	\end{align}
	and by taking expectations on both sides.
	By applying (\ref{eqn::ESestimationProofEXformula}), we get that
	\begin{align} 
	\label{eqn::ESestimationProofIntegralSquaredFormula}
	\int_{-\infty}^{q_\alpha} \int_{-\infty}^{q_\alpha} F(x) F(y) \mathrm{d}x \mathrm{d}y 
	= \left( \int_{-\infty}^{q_\alpha} F(x) \mathrm{d}x \right)^2
	= \left( \alpha q_\alpha  - \alpha \mathbb{E}\big[ Y \big| Y \le q_\alpha \big] \right)^2
	= \alpha^2 \big( q_\alpha - \xi_\alpha \big)^2.
	\end{align}
	Furthermore, notice that
	\begin{align}
		\begin{aligned} \label{eqn::ESestimationProofMinIntegralFormulaHelpStatement1}
			\int_{-\infty}^{q_\alpha} \int_{-\infty}^{q_\alpha} F(x \wedge y) \mathrm{d}x \mathrm{d}y
			= \int_{-\infty}^{q_\alpha} \int_{-\infty}^{y} F(x) \mathrm{d}x   \mathrm{d}y  + \int_{-\infty}^{q_\alpha} \int_{y}^{q_\alpha} F(y) \mathrm{d}x  \mathrm{d}y,
		\end{aligned}
	\end{align}	
	and by rearranging the order of integration for the first term in (\ref{eqn::ESestimationProofMinIntegralFormulaHelpStatement1}), we get that
	\begin{align}
	\begin{aligned} \label{eqn::ESestimationProofMinIntegralFormulaHelpStatement2}
	\int_{-\infty}^{q_\alpha} \int_{-\infty}^{y} F(x) \, \mathrm{d}x \mathrm{d}y
	&= \iint \limits_{\{ (x,y): \, y \le q_\alpha,\, x\le y \}}  F(x) \, \mathrm{d}x \mathrm{d}y
	= \iint \limits_{\{ (x,y): \, x \le q_\alpha, \, y \ge x \}}  F(x) \, \mathrm{d}y \mathrm{d}x  \\
	&= \int_{-\infty}^{q_\alpha} \int_{x}^{q_\alpha} F(x) \, \mathrm{d}y \mathrm{d}x
	= \int_{-\infty}^{q_\alpha}  F(x) (q_\alpha - x) \, \mathrm{d}x.
	\end{aligned}
	\end{align}	
	Thus, by first using (\ref{eqn::ESestimationProofMinIntegralFormulaHelpStatement1}) and (\ref{eqn::ESestimationProofMinIntegralFormulaHelpStatement2}) and by plugging in (\ref{eqn::ESestimationProofEXformula}) and (\ref{eqn::ESestimationProofIntegralSquaredFormula}), we obtain
	\begin{align}
	\begin{aligned}
	\label{eqn::ESestimationProofMinIntegralFormula}
	\int_{-\infty}^{q_\alpha} \int_{-\infty}^{q_\alpha} F(x \wedge y) \mathrm{d}x \mathrm{d}y
	&= 2 \int_{-\infty}^{q_\alpha} \int_{y}^{q_\alpha} F(y) \, \mathrm{d}x  \mathrm{d}y \\
	&= 2 \int_{-\infty}^{q_\alpha}  F(y) (q_\alpha - y) \, \mathrm{d}y \\
	&= 2 q_\alpha \int_{-\infty}^{q_\alpha}  F(y) \, \mathrm{d}y - 2 \int_{-\infty}^{q_\alpha} y  F(y) \, \mathrm{d}y \\
	&= 2 q_\alpha \big( \alpha q_\alpha - \alpha \xi_\alpha \big) + \alpha \mathbb{E} \big[ Y^2 \big| Y \le q_\alpha \big]  - \alpha q_\alpha^2 \\
	&= \alpha \mathbb{E} \big[ Y^2 \big| Y \le q_\alpha \big]  + \alpha q_\alpha^2 - 2\alpha q_\alpha \xi_\alpha.
	\end{aligned}
	\end{align}		
	Eventually, using (\ref{eqn::ESestimationProofIntegralSquaredFormula}) and (\ref{eqn::ESestimationProofMinIntegralFormula}), straight-forward calculations yield that
	\begin{align}
		&\frac{1}{\alpha^2} \int_{-\infty}^{q_\alpha} \int_{-\infty}^{q_\alpha} F(x \wedge y) - F(x) F(y) \mathrm{d}x \mathrm{d}y 
		= \frac{1}{\alpha} \operatorname{Var}(Y|Y\le q_\alpha) + \frac{1-\alpha}{\alpha} \big(q_\alpha - \xi_\alpha \big)^2,
	\end{align}
	which concludes the proof.
\end{proof}

\section{Separability of almost surely continuous functions}

\begin{definition}[Separability of a Stochastic Process]
	A stochastic process $\psi(x,\theta): \Omega \times \Theta \to \mathcal{Y}$ is called separable in the sense of Doob, if there exists in $\Omega$ an everywhere dense countable set $I$, and in $\Omega$ a nullset $N$ such that for any arbitrary open set $G \subset \Theta$ and every closed set $F \subset \mathcal{Y}$, the two sets
	\begin{align}
	&\{ x | \psi(x,\theta) \in F, \; \forall \theta \in G \} \qquad \mbox{and} \label{eqn::SeparabilitySet1} \\
	&\{ x | \psi(x,\theta) \in F, \; \forall \theta \in G \cap I \} \label{eqn::SeparabilitySet2}
	\end{align}
	differ from each other at most by a subset of $N$.
\end{definition}

\begin{proposition}[\cite{GikhmanSkorokhod2004}] \label{prop::SeparabilityContinuousFunctions}
	Let $\Theta$ and $\mathcal{Y}$ be metric spaces, $\Theta$ be a separable space. The sets (\ref{eqn::SeparabilitySet1}) and (\ref{eqn::SeparabilitySet2}) coincide for all $x \in \Omega$ for which the stochastic process $\psi(x,\theta)$ is continuous in $\theta$.
\end{proposition}

\begin{proof}
	It is clear that $\{ x | \psi(x,\theta) \in F, \; \forall \theta \in G \} \subseteq \{ x | \psi(x,\theta) \in F, \; \forall \theta \in G \cap I \}$. We thus only show the reverse. \par 
	Let $G \subset \Theta$ be an arbitrary open set and $F \subset \mathcal{Y}$ an arbitrary closed set. Let furthermore $x \in \Omega$ such that $\psi(x,\theta) \in F$ for all $\theta \in G \cap I$. We have to show that $\psi(x,\tilde \theta) \in F$ for all $\tilde \theta \in G$ but $\tilde \theta \notin I$. \par
	Thus, let $\tilde \theta \in G \setminus I$. Since $I$ is a dense set in $\Theta$, there exists a sequence $(\theta_n)_{n \in \mathbb{N}} \in \Theta \cap I$, such that $\theta_n \to \tilde \theta$ and since $G$ is an open set in $\Theta$ and $\tilde \theta \in G$, we can conclude that for $m \in \mathbb{N}$ large enough, $\theta_n \in G$ for all $n \ge m$. Furthermore, by continuity at $\theta$, it holds that $\psi(x,\theta_n) \to \psi(x,\tilde \theta)$ and since $\theta_n \in G \cap I$ for all $n$ large enough, $\psi(x,\theta_n) \in F$ by assumption. Eventually, since $F$ is a closed set, $\psi(x, \tilde \theta) \in F$ which proves the proposition.
\end{proof}

\begin{corollary}[Separability of continuous functions]
	\label{cor::SeparabilityContinuousFunctions}
	Let $\Theta$ and $\mathcal{Y}$ be metric spaces, $\Theta$ be a separable space, and let the stochastic process $\psi(x,\theta)$ be almost surely continuous. Then, $\psi$ is separable.
\end{corollary}

\begin{proof}
	Since $\psi(x,\theta)$ is continuous for all $x \in \Omega \setminus N$ for some $N \subset \Omega$ with $\mathbb{P}(N) = 0$. We get from Proposition \ref{prop::SeparabilityContinuousFunctions} that the sets (\ref{eqn::SeparabilitySet1}) and (\ref{eqn::SeparabilitySet2}) coincide for all $x \in \Omega \setminus N$, i.e. they differ only by a subset of~$N$.
\end{proof}

\end{appendices}

\addcontentsline{toc}{section}{References}
\setstretch{0.8}
\setlength{\bibsep}{5pt}
\bibliographystyle{apalike}	
\bibliography{../../Library/mybib}    

\begin{thebibliography}{}

\bibitem[Acerbi and Szekely, 2014]{Acerbi2014}
Acerbi, C. and Szekely, B. (2014).
\newblock {B}acktesting {E}xpected {S}hortfall.
\newblock {\em Risk}.

\bibitem[Andersen and Bollerslev, 1998]{Andersen1998}
Andersen, T. and Bollerslev, T. (1998).
\newblock {A}nswering the skeptics: {Y}es, standard volatility models do
  provide accurate forecasts.
\newblock {\em International Economic Review}, 39:885--905.

\bibitem[Andrews, 1994]{Andrews1994}
Andrews, D. (1994).
\newblock {E}mpirical {P}rocess {M}ethods in {E}conometrics.
\newblock In Engle, R. and McFadden, D., editors, {\em Handbook of
  Econometrics}, volume~4, chapter~37, pages 2247--2294. Elsevier.

\bibitem[Artzner et~al., 1999]{Artzner1999}
Artzner, P., Delbaen, F., Eber, J.-M., and Heath, D. (1999).
\newblock {C}oherent {M}easures of {R}isk.
\newblock {\em Mathematical Finance}, 9(3):203--228.

\bibitem[Barendse, 2017]{Barendse2017}
Barendse, S. (2017).
\newblock {Interquantile Expectation Regression}.
\newblock Available at \url{https://ssrn.com/abstract=2937665}.

\bibitem[{Basel Committee}, 2016]{Basel2016}
{Basel Committee} (2016).
\newblock {M}inimum capital requirements for {M}arket {R}isk.
\newblock Technical report, Basel Committee on Banking Supervision.
\newblock Available at \url{http://www.bis.org/bcbs/publ/d352.pdf}.

\bibitem[Bayer and Dimitriadis, 2017a]{ESRegPackge}
Bayer, S. and Dimitriadis, T. (2017a).
\newblock {\em esreg: {J}oint ({V}a{R}, {ES}) {R}egression}.
\newblock R package version 0.2.0, available at
  \url{https://github.com/BayerSe/esreg}.

\bibitem[Bayer and Dimitriadis, 2017b]{BayerDimi2017}
Bayer, S. and Dimitriadis, T. (2017b).
\newblock {R}egression-based {E}xpected {S}hortfall {B}acktesting.
\newblock Working Paper.

\bibitem[Bollerslev, 1986]{Bollerslev1986}
Bollerslev, T. (1986).
\newblock {G}eneralized autoregressive conditional heteroskedasticity.
\newblock {\em Journal of Econometrics}, 31(3):307--327.

\bibitem[Brazauskas et~al., 2008]{Brazauskas2008}
Brazauskas, V., Jones, B.~L., Puri, M.~L., and Zitikis, R. (2008).
\newblock {E}stimating conditional tail expectation with actuarial applications
  in view.
\newblock {\em Journal of Statistical Planning and Inference},
  138(11):3590--3604.

\bibitem[Chen, 2008]{Chen2008}
Chen, S.~X. (2008).
\newblock {N}onparametric {E}stimation of {E}xpected {S}hortfall.
\newblock {\em Journal of Financial Econometrics}, 6(1):87--107.

\bibitem[Chernozhukov and Umantsev, 2001]{Chernozhukov2001}
Chernozhukov, V. and Umantsev, L. (2001).
\newblock {C}onditional value-at-risk: {A}spects of modeling and estimation.
\newblock {\em Empirical Economics}, 26(1):271--292.

\bibitem[Corsi, 2009]{Corsi2009}
Corsi, F. (2009).
\newblock {A Simple Approximate Long-Memory Model of Realized Volatility}.
\newblock {\em Journal of Financial Econometrics}, 7(2):174--196.

\bibitem[Dimitriadis and Bayer, 2017]{DimiBayer2017Supplement}
Dimitriadis, T. and Bayer, S. (2017).
\newblock {Online Supplement for ``A Joint Quantile and Expected Shortfall
  Regression Framework''}.

\bibitem[Efron, 1979]{Efron1979}
Efron, B. (1979).
\newblock {B}ootstrap {M}ethods: {A}nother {L}ook at the {J}ackknife.
\newblock {\em The Annals of Statistics}, 7(1):1--26.

\bibitem[Efron, 1991]{Efron1991}
Efron, B. (1991).
\newblock {R}egression percentiles using asymmetric squared error loss.
\newblock {\em Statistica Sinica}, 1:93--125.

\bibitem[Ehm et~al., 2016]{Ehm2016}
Ehm, W., Gneiting, T., Jordan, A., and Krüger, F. (2016).
\newblock {O}f quantiles and expectiles: consistent scoring functions,
  {C}hoquet representations and forecast rankings.
\newblock {\em Journal of the Royal Statistical Society: Series B (Statistical
  Methodology)}, 78(3):505--562.

\bibitem[Engle and Manganelli, 2004]{Engle2004}
Engle, R. and Manganelli, S. (2004).
\newblock {CAV}ia{R}: {C}onditional {A}utoregressive {V}alue at {R}isk by
  {R}egression {Q}uantiles.
\newblock {\em {Journal of Business and Economic Statistics}}, 22(4):367--381.

\bibitem[Fissler, 2017]{FisslerThesis2017}
Fissler, T. (2017).
\newblock {\em {O}n {H}igher {O}rder {E}licitability and {S}ome {L}imit
  {T}heorems on the {P}oisson and {W}iener {S}pace}.
\newblock PhD thesis, Universit\"at Bern.

\bibitem[Fissler and Ziegel, 2016]{Fissler2016}
Fissler, T. and Ziegel, J.~F. (2016).
\newblock {H}igher order elicitability and {Osband's} principle.
\newblock {\em Annals of Statistics}, 44(4):1680--1707.

\bibitem[Fissler et~al., 2016]{Fissler2016b}
Fissler, T., Ziegel, J.~F., and Gneiting, T. (2016).
\newblock {E}xpected {S}hortfall is jointly elicitable with {V}alue at {R}isk -
  {I}mplications for backtesting.
\newblock {\em Risk Magazine}, Janaury 2016.

\bibitem[Gaglianone et~al., 2011]{Gaglianone2011}
Gaglianone, W.~P., Lima, L.~R., Linton, O., and Smith, D.~R. (2011).
\newblock {E}valuating {V}alue-at-{R}isk {M}odels via {Q}uantile {R}egression.
\newblock {\em Journal of Business \& Economic Statistics}, 29(1):150--160.

\bibitem[Gikhman and Skorokhod, 2004]{GikhmanSkorokhod2004}
Gikhman, I. and Skorokhod, A. (2004).
\newblock {\em {T}he {T}heory of {S}tochastic {P}rocesses {I}}, volume 210 of
  {\em Classics in Mathematics}.
\newblock Springer Berlin Heidelberg.

\bibitem[Gneiting, 2011]{Gneiting2011b}
Gneiting, T. (2011).
\newblock {M}aking and {E}valuating {P}oint {F}orecasts.
\newblock {\em Journal of the American Statistical Association},
  106(494):746--762.

\bibitem[Gourieroux and Monfort, 1995]{Gourieroux1995}
Gourieroux, C. and Monfort, A. (1995).
\newblock {\em {S}tatistics and {E}conometric {M}odels: {V}olume 1, {G}eneral
  {C}oncepts, {E}stimation, {P}rediction and {A}lgorithms}.
\newblock Cambridge University Press.

\bibitem[Halbleib and Pohlmeier, 2012]{Halbleib2012}
Halbleib, R. and Pohlmeier, W. (2012).
\newblock {Improving the value at risk forecasts: Theory and evidence from the
  financial crisis}.
\newblock {\em Journal of Economic Dynamics and Control}, 36(8):1212--1228.

\bibitem[Hall and Sheather, 1988]{Hall1988}
Hall, P. and Sheather, S.~J. (1988).
\newblock {O}n the {D}istribution of a {S}tudentized {Q}uantile.
\newblock {\em Journal of the Royal Statistical Society. Series B
  (Methodological)}, 50(3):381--391.

\bibitem[Hendricks and Koenker, 1992]{Hendricks1992}
Hendricks, W. and Koenker, R. (1992).
\newblock {H}ierarchical {S}pline {M}odels for {C}onditional {Q}uantiles and
  the {D}emand for {E}lectricity.
\newblock {\em Journal of the American Statistical Association},
  87(417):58--68.

\bibitem[Huber, 1967]{Huber1967}
Huber, P. (1967).
\newblock {T}he behavior of maximum likelihood estimates under nonstandard
  conditions.
\newblock In {\em Proceedings of the Fifth Berkeley Symposium on Mathematical
  Statistics and Probability}, pages 221--233. Berkeley: University of
  California Press.

\bibitem[Koenker, 1994]{Koenker1994}
Koenker, R. (1994).
\newblock {C}onfidence {I}ntervals for {R}egression {Q}uantiles.
\newblock In Mandl, P. and Hu{\v{s}}kov{\'a}, M., editors, {\em Asymptotic
  Statistics: Proceedings of the Fifth Prague Symposium, held from September
  4--9, 1993}, pages 349--359. Physica-Verlag Heidelberg.

\bibitem[Koenker, 2005]{Koenker2005book}
Koenker, R. (2005).
\newblock {\em {Q}uantile {R}egression}.
\newblock Econometric Society Monographs. Cambridge University Press.

\bibitem[Koenker and Machado, 1999]{Koenker1999}
Koenker, R. and Machado, J. A.~F. (1999).
\newblock {G}oodness of {F}it and {R}elated {I}nference {P}rocesses for
  {Q}uantile {R}egression.
\newblock {\em Journal of the American Statistical Association},
  94(448):1296--1310.

\bibitem[Koenker and Xiao, 2006]{KoenkerXiao2006}
Koenker, R. and Xiao, Z. (2006).
\newblock {Q}uantile {A}utoregression.
\newblock {\em Journal of the American Statistical Association},
  101(475):980--990.

\bibitem[Komunjer, 2013]{Komunjer2013}
Komunjer, I. (2013).
\newblock {Q}uantile {P}rediction.
\newblock In {\em Handbook of Economic Forecasting}, volume~2, chapter~17,
  pages 961--994. Elsevier.

\bibitem[Lambert et~al., 2008]{Lambert2008}
Lambert, N.~S., Pennock, D.~M., and Shoham, Y. (2008).
\newblock {E}liciting {P}roperties of {P}robability {D}istributions.
\newblock In {\em Proceedings of the 9th ACM Conference on Electronic
  Commerce}, pages 129--138. ACM.

\bibitem[Louren{\c{c}}o et~al., 2003]{Lourenco2003}
Louren{\c{c}}o, H.~R., Martin, O.~C., and St{\"u}tzle, T. (2003).
\newblock {I}terated {L}ocal {S}earch.
\newblock In Glover, F. and Kochenberger, G.~A., editors, {\em Handbook of
  Metaheuristics}, pages 320--353. Springer US, Boston, MA.

\bibitem[Nadarajah et~al., 2014]{Nadarajah2014}
Nadarajah, S., Zhang, B., and Chan, S. (2014).
\newblock {E}stimation methods for expected shortfall.
\newblock {\em Quantitative Finance}, 14(2):271--291.

\bibitem[Nelder and Mead, 1965]{Nelder1965}
Nelder, J.~A. and Mead, R. (1965).
\newblock {A} {S}implex {M}ethod for {F}unction {M}inimization.
\newblock {\em The Computer Journal}, 7(4):308--313.

\bibitem[Newey and McFadden, 1994]{NeweyMcFadden1994}
Newey, W. and McFadden, D. (1994).
\newblock {L}arge sample estimation and hypothesis testing.
\newblock In Engle, R. and McFadden, D., editors, {\em Handbook of
  Econometrics}, volume~4, chapter~36, pages 2111--2245. Elsevier.

\bibitem[Nolde and Ziegel, 2017]{Nolde2017}
Nolde, N. and Ziegel, J.~F. (2017).
\newblock {E}licitability and backtesting: {P}erspectives for banking
  regulation.
\newblock arXiv:1608.05498 [q-fin.RM].

\bibitem[Taylor, 2008a]{Taylor2008}
Taylor, J.~W. (2008a).
\newblock {E}stimating {V}alue at {R}isk and {E}xpected {S}hortfall {U}sing
  {E}xpectiles.
\newblock {\em Journal of Financial Econometrics}, 6(2):231--252.

\bibitem[Taylor, 2008b]{Taylor2008b}
Taylor, J.~W. (2008b).
\newblock {U}sing {E}xponentially {W}eighted {Q}uantile {R}egression to
  {E}stimate {V}alue at {R}isk and {E}xpected {S}hortfall.
\newblock {\em Journal of Financial Econometrics}, 6(3):382--406.

\bibitem[Taylor, 2017]{Taylor2017}
Taylor, J.~W. (2017).
\newblock {F}orecasting {V}alue at {R}isk and {E}xpected {S}hortfall {U}sing a
  {S}emiparametric {A}pproach {B}ased on the {A}symmetric {L}aplace
  {D}istribution.
\newblock {\em Forthcoming in Journal of Business and Economic Statistics}.

\bibitem[{van der Vaart}, 1998]{VanderVaart1998}
{van der Vaart}, A.~W. (1998).
\newblock {\em {A}symptotic statistics}.
\newblock Cambridge Series in Statistical and Probabilistic Mathematics.
  Cambridge University Press.

\bibitem[\v{Z}ike\v{s} and Barun\'{i}k, 2016]{Zikes2016}
\v{Z}ike\v{s}, F. and Barun\'{i}k, J. (2016).
\newblock {S}emi-parametric {C}onditional {Q}uantile {M}odels for {F}inancial
  {R}eturns and {R}ealized {V}olatility.
\newblock {\em Journal of Financial Econometrics}, 14(1):185--226.

\bibitem[Weber, 2006]{Weber2006}
Weber, S. (2006).
\newblock {Distribution Invariant Risk Measures, Information, and Dynamic
  Consistency}.
\newblock {\em Mathematical Finance}, 16(2):419--441.

\bibitem[Xiao et~al., 2015]{Xiao2015}
Xiao, Z., Guo, H., and Lam, M.~S. (2015).
\newblock {Q}uantile {R}egression and {V}alue at {R}isk.
\newblock In Lee, C.-F. and Lee, J.~C., editors, {\em Handbook of Financial
  Econometrics and Statistics}, pages 1143--1167. Springer.

\bibitem[Ziegel et~al., 2017]{Ziegel2017}
Ziegel, J.~F., Krüger, F., Jordan, A., and Fasciati, F. (2017).
\newblock {M}urphy {D}iagrams: {F}orecast {E}valuation of {E}xpected
  {S}hortfall.
\newblock arXiv:1705.04537 [q-fin.RM].

\bibitem[Zwingmann and Holzmann, 2016]{Zwingmann2016}
Zwingmann, T. and Holzmann, H. (2016).
\newblock {Asymptotics for the expected shortfall}.
\newblock arXiv:1611.07222 [math.ST].

\end{thebibliography}

\end{document}